\documentclass[a4paper,11pt]{article}

\usepackage{amsthm,amssymb,amsmath,graphicx,mathtools,enumerate}
\usepackage[hidelinks]{hyperref}

\usepackage{caption}
\usepackage{subcaption}
\usepackage[shortlabels]{enumitem} 
\usepackage[british]{babel}
\usepackage{setspace}
\usepackage{thm-restate}
\usepackage{cleveref}

\usepackage{framed}
\usepackage{xspace}
\usepackage{comment}
\usepackage{hyperref}

\usepackage[margin = 2.5cm]{geometry}

\usepackage[square,sort,comma,numbers]{natbib} 
\usepackage{tikz}
\usetikzlibrary{positioning}
\usetikzlibrary{backgrounds}
\usetikzlibrary{shapes}

\theoremstyle{plain}
\newtheorem*{thm*}{Theorem}
\newtheorem{thm}{Theorem}[section]
\Crefname{thm}{Theorem}{Theorems}

\newtheorem*{lem*}{Lemma}
\newtheorem{lem}[thm]{Lemma}
\Crefname{lem}{Lemma}{Lemmas}

\newtheorem*{claim*}{Claim}
\newtheorem{claim}[thm]{Claim}
\crefname{claim}{Claim}{Claims}
\Crefname{claim}{Claim}{Claims}

\newtheorem{prop}[thm]{Proposition}
\Crefname{prop}{Proposition}{Propositions}

\newtheorem{cor}[thm]{Corollary}
\Crefname{cor}{Corollary}{Corollaries}
\crefname{cor}{Corollary}{Corollaries}

\Crefname{conj}{Conjecture}{Conjectures}

\Crefname{qn}{Question}{Questions}

\Crefname{obs}{Observation}{Observations}

\theoremstyle{definition}

\Crefname{prob}{Problem}{Problems}

\newtheorem*{fact*}{Fact}
\Crefname{fact}{Fact}{Facts}

\newtheorem*{ex*}{Example}
\Crefname{ex}{Example}{Examples}

\theoremstyle{definition}

\newtheorem{defn}[thm]{Definition}
\newtheorem*{defn*}{Definition}
\Crefname{defn}{Definition}{Definitions}

\theoremstyle{remark}
\newtheorem*{rem*}{Remark}
\Crefname{rem}{Remark}{Remarks}

\numberwithin{equation}{section}

\DeclareMathOperator{\Bin}{Bin}

\newcommand{\eps}{\varepsilon}
\newcommand{\sm}{\setminus}
\renewcommand{\subset}{\subseteq}

\newcommand{\ceil}[1]{\left\lceil #1 \right \rceil}

\newcommand{\C}{\mathcal{C}}

\newcommand{\HH}{\mathcal{H}}

\newcommand{\G}{\mathcal{G}}
\newcommand{\M}{\mathcal{M}}
\newcommand{\Z}{\mathbb{Z}}
\newcommand{\PP}{\mathcal{P}}
\newcommand{\Zp}{\Z^{\ge 0}}
\newcommand{\om}{\omega}
\newcommand{\calA}{\mathcal{A}}
\newcommand{\calB}{\mathcal{B}}

\newcommand{\Cbad}{\C^b}
\newcommand{\Ch}{\C^{\HH}}
\newcommand{\Cn}{\C^0}
\newcommand{\Cs}{\C^s}
\newcommand{\bnd}{t} 

\newcommand{\compt}{\Phi}  
\newcommand{\ccompt}{\Psi} 
\newcommand{\qinit}{400r+2} 
\newcommand{\qrest}{(1/\mu + 100)r^2}
\newcommand{\qcompt}{(1/\mu)r^2}
\newcommand{\qbad}{100r^2}  
\newcommand{\robmat}{robustly $2$-matchable\xspace}

\usepackage{framed}

\title{Minimum degree conditions for monochromatic cycle partitioning}

\author{
	D\'aniel Kor\'andi\thanks{Mathematical Institute, University of Oxford, Andrew Wiles Building, Radcliffe Observatory Quarter, Woodstock Road, Oxford, United Kingdom. Research supported in part by SNSF grants 200020-162884 and 200021-175977, and an SNSF Postdoc.Mobility fellowship. Email: \texttt{daniel.korandi}@\texttt{maths.ox.ac.uk}.}
    \and
	Richard Lang\thanks{Institute for Computer Science, 
		University of Heidelberg,
		Im Neuenheimer Feld 205,
		69120 Heidelberg, Germany.
		The research leading to these results was partially supported by the Deutsche Forschungsgemeinschaft (DFG, German Research Foundation) -- 42821240 (R. Lang). 
		Email: \texttt{lang}@\texttt{informatik.uni-heidelberg.de}.}
    \and
    Shoham Letzter\thanks{
		Department of Mathematics, 
		University College London, 
		Gower Street, London, WC1E 6BT, UK. 
		Email: \texttt{s.letzter}@\texttt{ucl.ac.uk}. 
		The research was supported by Dr.\ Max R\"ossler, the Walter Haefner Foundation, the ETH Zurich Foundation and the Royal Society.
    }
    \and
    Alexey Pokrovskiy\thanks{%
		Department of Mathematics, 
		University College London, 
		Gower Street, London, WC1E 6BT, UK.
		Email: \texttt{Dr.Alexey.Pokrovskiy}@\texttt{gmail.com}.%
	}
}
\date{}

\begin{document}

\maketitle

\begin{abstract}
	\setstretch{1.1}	
	\setlength{\parskip}{\medskipamount}
    \setlength{\parindent}{0pt}
    \noindent
	A classical result of Erd{\H o}s, Gy{\'a}rf{\'a}s and Pyber states that any $r$-edge-coloured complete graph has a partition into $O(r^2 \log r)$ monochromatic cycles.
  	Here we determine the minimum degree threshold for this property.
  	More precisely, we show that there exists a constant $c$ such that any $r$-edge-coloured graph on $n$ vertices with minimum degree at least $n/2 + c \cdot r\log n$ has a partition into $O(r^2)$ monochromatic cycles.
  	We also provide constructions showing that the minimum degree condition and the number of cycles are essentially tight.
\end{abstract}

\setstretch{1.2}

\section{Introduction}
	Monochromatic cycle partitioning is a combination of Ramsey-type and covering problems. Given an edge-coloured host graph $G$, one seeks to partition the vertex set of $G$ into as few monochromatic cycles as possible.\footnote{To avoid some trivial cases, we consider the empty set, vertices and edges to be cycles.}
	The case where the number of used cycles $f$ can be upper-bounded by a function of the number of colours $r$ is of particular interest.
	A classical result in this area is due to Erd{\H o}s, Gy{\'a}rf{\'a}s and Pyber~\cite{EGP91}, who showed that any $r$-edge-coloured complete graph $G=K_n$ admits a partition into $\ceil{25r^2 \log r}$ monochromatic cycles.
	The same authors conjectured that their result could in fact be improved to $r$ cycles.
	For $r=2$, this had been suggested about 20 years earlier by Lehel in a stronger sense, i.e.\ with the cycles having distinct colours.
	Lehel's conjecture was first proved for large $n$ by {\L}uczak, R{\"o}dl and Szemer{\'e}di~\cite{LRS98} and then for all $n$ by Bessy and Thomass{\'e}~\cite{BT10}, after preliminary work by Gy{\'a}rf{\'a}s~\cite{Gya83}.
	For $r \geq 3$, the conjecture of Erd{\H o}s, Gy{\'a}rf{\'a}s and Pyber turned out to be false.
	Pokrovskiy~\cite{Pok14} provided colourings of the complete graph that require $r$ cycles and a single additional vertex for a partition.
	He conjectured, however, that a partition into $r$ cycles and a constant number of vertices $c(r)$ should nevertheless be sufficient.
	In support of his conjecture, Pokrovskiy showed the case of $r=3$ with $c(r)=43000$.
	This was independently confirmed by Letzter~\cite{Let18} with $c(r) = 60$.
	The best known general upper bound for the number of monochromatic cycles required to partition any $r$-coloured complete graph is $100 r \log r$, established by  Gy{\'a}rf{\'a}s, Ruszink{\'o}, S{\'a}rk{\"o}zy and Szemer{\'e}di~\cite{GRSS06}.
	
	In the past decades monochromatic partitions of the complete graph have been researched in many ways, such as partitioning into graphs other than cycles~\cite{SS00,GS16-bounded}, more general colourings~\cite{CS16} and partitions of hypergraphs~\cite{BCFP19,BHS19,GS12,GLL+20}.
	For a broader overview, we refer the reader to the recent survey of~Gy{\'a}rf{\'a}s~\cite{Gya16}.
	Another natural problem arises when we consider  host graphs   that need not be complete.
	In particular, for which families of graphs can we still partition the vertex set into few monochromatic cycles?
	This question has been investigated for complete bipartite graphs~\cite{Hax97}, graphs with fixed independence number~\cite{Sar11}, infinite graphs~\cite{Rad78,ESSS17} and random graphs~\cite{KMN+18,LL18} among others.
	Here we are interested in families of graphs characterized by a large minimum degree.
	
	The study of minimum degree conditions for spanning substructures has a long tradition in extremal graph theory, Dirac's theorem being a classical example.
	Recent milestones of this area include the resolution of the P{\'o}sa-Seymour conjecture by Koml{\'o}s, S{\'a}rk{\"o}zy and Szemer{\'e}di~\cite{KSS98}, 
	the Bandwidth theorem by B{\"o}ttcher, Schacht and Taraz~\cite{BST09}, and the Hamilton decomposition theorem by Csaba,   K{\"u}hn,  Lo,   Osthus  and Treglown~\cite{CKL+16}.
	Many other results in this line of research are covered in the survey of K{\"u}hn and Osthus~\cite{KO14}.
	
	For monochromatic cycle partitions, the research of minimum degree conditions was initiated by Balogh,   Bar\'{a}t,   Gerbner, Gy\'arf\'as  and S\'ark\"ozy~\cite{BBG+14} with a strengthening of Lehel's conjecture. They showed that every 2-edge-coloured graph $G$ on $n$ vertices of minimum degree $(3/4 + \eps)n$ admits a partition of all but $o(n)$ vertices into two monochromatic cycles of distinct colours. They also conjectured that this can be improved to a proper partition even without the term of $\eps n$.
	(An easy construction shows that this is best possible.)
	The extension to a proper partition was verified by DeBiasio and Nelsen~\cite{DN17} and the full conjecture was subsequently proved by Letzter~\cite{Let15}. Given these advances, Pokrovskiy~\cite{Pok16} conjectured that for a 2-edge-coloured graph $G$ with $\delta(G) \geq 2n/3$ and $\delta(G) \geq n/2$ a partition into 3 and 4, respectively, cycles is possible.
	(Again, constructions show  that this is essentially best possible.)
	Allen, B{\"o}ttcher, Lang, Skokan and Stein~\cite{ABL+20} confirmed the first part of this conjecture approximately, i.e.\ for $\delta(G) \geq (2/3 + \eps)n$. Thus the problem for two colours is increasingly well understood.
	
	The goal of this research was to determine the minimum degree threshold for partitioning an $r$-edge-coloured graph into $f(r)$ monochromatic cycles for general $r$, for any function $f(r)$ that depends only on $r$.
	A lower bound of $n/2$ for this threshold is shown by the simple example of a slightly unbalanced complete bipartite graph.
	However, a more involved construction shows that a minimum degree below $n/2 + O(\log n / \log \log n)$ already requires $\Omega(\log n / \log \log n)$ monochromatic cycles for a partition.

 	\begin{prop}\label{prop:degree-best-possible}
		Let $n$ be sufficiently large. Then there is a $2$-edge-coloured graph $G$ on $n$ vertices with $\delta(G) \geq {n}/{2} + {\log n}/{(16\log\log n)}$ whose vertices cannot be partitioned into fewer than $ {\log n}/{(32\log\log n)}$  monochromatic cycles.
	\end{prop}

	Our main contribution states that a minimum degree slightly larger than this is in turn sufficient for a partition into $O(r^2)$ cycles.

	\begin{thm}\label{thm:main}
		For $r \geq 2$, let $n$ be sufficiently large. Then any $r$-edge-coloured graph $G$ on $n$ vertices with  $ \delta(G) \geq n/2 + 1200 r\log n$ admits a  partition into $10^7 r^2$ monochromatic cycles.
	\end{thm}

	We also provide a construction that shows that the number of cycles of \Cref{thm:main} is essentially best possible.
	
 	\begin{prop} \label{prop:cyclelowerbound}
		Let $\eps>0$ and $r$ be sufficiently large. Then there is an $r$-edge-coloured graph $G$ on $n=n(\eps,r)$ vertices with $\delta(G)\ge (1-4\eps)n$, whose vertices cannot be covered by fewer than $ {\eps^2(r-1)^2}/{4}$ monochromatic trees.
	\end{prop}

	It is worth mentioning here a recent paper of Buci\'c, Kor\'andi and Sudakov \cite{BKS19}, who were interested in covering $r$-coloured random graphs $G(n,p)$ by monochromatic trees. Similarly to our results, they proved that the minimum number of monochromatic trees needed is $\Theta(r^2)$ when $p$ is just above the threshold for the existence of a covering with a bounded number of trees. 

	Our results imply in particular, that we can determine the smallest number of cycles necessary for a partition of bounded degree graphs up to a constant factor. This stands in contrast to the situation for complete graphs, where the gap between upper and lower bound remains a factor of $\log r$.

\section{Overview} \label{sec:overview}

	A brief overview of the proof of \Cref{thm:main} is as follows. 
	Using Szemer\'edi's regularity lemma, we obtain a regular partition $\{V_0,V_1,\dots,V_m\}$ of the vertices of $G$, and define the corresponding reduced graph $\G$. 
	We then select $O(r^2)$ monochromatic components of $\G$ in such a way that their union, denoted by $\HH$, robustly contains a perfect matching.
	The robustness roughly translates into $\HH$ having a perfect matching even after removing any small set of vertices.
	So, in particular, $\HH$ has a perfect matching $\M$.
	We can turn $\M$ into $O(r^2)$ disjoint monochromatic cycles $\Ch$ covering almost all of $G$ using a method of {\L}uczak~\cite{Luc99}.
	
	The plan is now to add the remaining vertices $V(G) \sm \Ch$ into the cycles of $\Ch$.
	More precisely, we intend to use the blow-up lemma to find monochromatic spanning paths in the regular pairs corresponding to $\M$.
	There are two obstacles to this.
	First, there might be a small number of ``bad'' vertices blocking the use of the blow-up lemma.
	Second, the clusters $V_i \sm \Cbad$ might be slightly different in size, which prevents us from even allocating spanning paths in the pairs.
	
	We deal with the irregular vertices by covering them with $O(r^2)$ additional cycles $\Cbad$, exploiting their large degrees.
	We then balance the clusters by carefully extending the cycles of $\Ch$.
	At this point, the robustness under which $\HH$ has a perfect matching is crucial.
	Having overcome these two issues, we can finish by applying the blow-up lemma to add the remaining vertices into $\Ch$.
	Thus $\Ch \cup \Cbad$ presents the desired cycle partition.
	
	This method works as long as $\G$ admits a spanning subgraph $\HH$, which  robustly contains a perfect matching, but unfortunately, we cannot always guarantee this. However, if such subgraph $\HH$ does not exist, then we can show that $G$ must be balanced bipartite after the removal of $O(r)$ monochromatic cycles $\C$. (At this point, we use the additional $1200 r \log n$ in the minimum degree.) Thus we can apply a bipartite analogue of the above detailed approach to cover the rest with cycles $\Ch \cup \Cbad$. In this case the cycle partition consists of $\Ch \cup \Cbad \cup \C$.

\section{Preliminaries}\label{sec:preliminaries}
	In this section we introduce some notations and tools needed for the proof of \Cref{thm:main}.

	\subsection{Notation}
		Let $G=(V,E)$ be a graph.  The \emph{order} of $G$ is $|V(G)|$ and the \emph{size} of $G$ is $|E(G)|$.
		We denote the neighbourhood of a vertex $v$ by $N_G(v)$ and write $N_G(v,W) = N_G(v) \cap W$ for a set of vertices $W \subset V(G)$.  We denote the degree of $v$ by $\deg_G(v) = |N_G(v)|$ and $\deg_G(v,W) = |N_G(v,W)|$.
		For a set of vertices $S \subset G$ we write $N_G(S) = \left(\bigcup_{s \in S} N(s)\right) \sm S$.
		When the underlying graph $G$ is clear from the context, we often omit the subscript $G$.
		For another graph $H$, the union $G\cup H$ is the graph on vertex set $V(G) \cup V(H)$ and edge set $E(G) \cup E(H)$.
		The independence number of $G$ is denoted by $\alpha(G)$.
		For disjoint sets $X,Y \subset V(G)$, we denote by $G[X,Y]$ the induced bipartite subgraph of $G$ with bipartition $\{X,Y\}$.
		
		An $r$-edge-colouring of $G$ assigns 
		one \emph{colour} from the set $[r]=\{1,2,\dots, r\}$ to each edge of $G$.  For $i \in [r]$, we use $G_i$ to denote the subgraph of $G$ whose edges are those that have colour $i$. 
		A (connectivity) component of colour $i$ of $G$ is a (connectivity) component of $G_i$.
		When the context is clear, we will simplify our notation by using $i$ instead of $G_i$ in the subscript of $N$ and $\deg$.
		For example, a vertex $v$ with $\deg_i(v) = 0$ is in an $i$-coloured component of order 1.
			
		A $v$-$w$-path is a path that starts at $v$ and ends at $w$.
		As said above, we allow the empty set, single vertices and edges in our cycle partitions. We occasionally use the term ``proper cycle'' to emphasize that a cycle is not an empty set, a vertex or an edge.

		In some of our statements, we will make assumptions of the form $x \ll y$ for certain parameters $x$ and $y$. This should be understood as equivalent to the condition $x\le f(y)$ for some unspecified increasing function $f$. In our usage, this is always a strengthening of $x\le y$.
		
	\subsection{Regularity}\label{sec:regularity}
		Given a graph $G$ and disjoint vertex sets $V, W \subset V(G)$ we
		denote the number of edges between $V$ and $W$ by $e(V,W)$ and the
		density of $(V,W)$ by $d(V,W) = e(V,W)/(|V||W|)$.  The pair $(V,W)$
		is called $\eps$-regular, if all subsets $X
		\subset V$, $Y \subset W$ with $|X| \ge \eps |V|$ and $|Y| \ge \eps
		|W|$ satisfy $|d(V,W)-d(X,Y)| \le \eps$.  

		We say that a
		vertex $v \in V$ has \emph{typical degree} in $(V,W)$, if $\deg(v,W)
		\ge (d(V,W)-\eps)|W|$.  It follows directly from the definition of
		$\eps$-regularity that
		\begin{equation}\label{equ:typical-degree-in-reg-pair}
			\text{all but at most $\eps |V|$ vertices in $V$ have typical degree in $(V,W)$.}
		\end{equation}	
		
		The next lemma allows us to find (spanning) paths in regular pairs.
		A similar statement can be deduced from the well-known blow-up lemma~\cite{KSS97}, but it can also be proved independently, see \Cref{appendix:proofs-prelims}.
		
		\begin{lem}[Long paths in regular pairs] \label{lem:inner-paths}
			Let $n$ be an integer and let $\eps, d$ be numbers with $0 < 1/n
			\ll \eps \ll d < 1$.  Suppose that $(V_1,V_2)$ is an
			$\eps$-regular pair of density $d=d(V_1,V_2)$ and with $|V_1| =
			|V_2| = n$ in a graph $G$. For $i \in \{1, 2\}$, let $v_i \in V_i$ and let $U_i \subset V_i$ be a set of size at least $n/6$ which contains at least $2\eps n$ neighbours of $v_{3-i}$.
			
			Then for every $2 \le k \le (1 - 24\eps)\cdot \min\{|U_1 |, |U_2|\}$, there is a $v_1$-$v_2$-path of order $2k$ in $G[U_1\cup \{v_1\}, U_2\cup \{v_2\}]$.
			
			If, additionally, $\delta(G[U_1, U_2]) \ge 5\eps n$, then $G[U_1\cup \{v_1\}, U_2\cup \{v_2\}]$ contains a $v_1$-$v_2$-path of order $2k$ for every $2 \le k \le \min\{|U_1\cup\{v_1\}|, |U_2\cup\{v_2\}|\}$.
		\end{lem}

		Szemer\'edi's Regularity Lemma~\cite{Sze76} allows one to partition
		the vertex set of
		a graph into clusters of vertices, in a way that most pairs of
		clusters are regular.  We will use the regularity lemma
		in its degree form (see~\cite{KS96}), with $r$ colours and a prepartition.

		\begin{lem}[Regularity Lemma]\label{lem:regularity}
			For every $\eps>0$ and integers $r,\ell$ there is an $M =
			M(\eps,r,\ell)$ such that the following holds.
			Let $G$ be a graph on $n\ge {1}/{\eps}$ vertices whose edges are coloured with $r$ colours, let $\{W_1,\dots,W_{\ell'}\}$ be an equipartition of $V(G)$ for some $1\le \ell'\le \ell$, and let $d > 0$.
			Then there is a partition $\{V_0, \dots, V_m\}$ of $V(G)$ and a subgraph $G'$ of $G$ with vertex set $V(G) \sm V_0$ such that the following conditions hold.
			\begin{enumerate}[\upshape (a)]
				\item 
					$1/\eps \le m \le M$,
				\item \label{itm:reg-lemma-sizes}
					$|V_0| \le \eps n$ and $|V_1| = \dots = |V_m| \le \eps n$,
				\item 
					for every $i \in [m]$, there is $j \in [\ell']$ with $V_i \subset W_j$,
				\item 
					for every $j \in [\ell']$, there are equally many $i \in [m]$ with $V_i \subset W_j$,
				\item $\deg_{G'}(v) \ge \deg_G(v)-(rd+\eps)n$ for each $v
					\in V(G) \sm V_0$,
				\item 
					$G'[V_i] $ contains no edges for $i \in [m]$, and
				\item \label{itm:reg-lemma-eps-regular} 
					all pairs $(V_i,V_j)$ are $\eps$-regular in $G'$ and have in each colour either density $0$ or at least $d$.
			\end{enumerate}
		\end{lem}

		Let $G$ be an $r$-edge-coloured graph with a partition $\{V_0,
		\dots,V_m\}$ obtained from \Cref{lem:regularity} with parameters
		$\eps$ and $d$.
		We define the $(\eps,d)$-\emph{reduced graph} $\G$ to be a graph with vertex set $V(\G) = \{x_1, \dots,x_m\}$ where two vertices $x_i$ and $x_j$ are connected by an edge of colour
		$c$, if $(V_i,V_j)$ is an $\eps$-regular pair of density at least $d$
		in colour $c$ (if this holds for multiple colours, we choose one of them arbitrarily).
		Note that if $G$ was balanced $\ell$-partite with partition $\{W_1,\dots,W_\ell\}$, then $\G$ is a balanced $\ell$-partite graph, as well.
		It is often convenient to refer to a \emph{cluster} $V_i$ via its
		corresponding vertex in the reduced graph, i.e.\ $V_i = V(x_i)$.
		
		The following properties of the reduced graph are easy to check. We provide a proof in \Cref{appendix:proofs-prelims}.
		\begin{prop} \label{prop:reducedgraph}
			\hfil
			\begin{enumerate}[\upshape (a)]
				\item \label{itm:reducedgraph-a} 
					If $\deg_G(v)\ge cn$ for some $v\in V_i$, $i\in[m]$, then $\deg_{\G}(x_i)\ge (c-rd-\eps)m$.
				\item \label{itm:reducedgraph-b}
					If $\deg_G(v)\ge cn$ for all but $\eta n$ vertices $v\in V(G)$, then $\deg_{\G}(x)\ge (c-rd-\eps)m$ for all but $(\eta+\eps) m$ vertices $x \in V(\G)$.
				\item \label{itm:reducedgraph-c}
					If $\bigcup_{x_i\in X} V_i$ induces at least $cn^2$ edges in $G$ for some $X\subset V(\G)$, then $X$ induces at least $(c-rd-\eps)m^2$ edges.
			\end{enumerate}
		\end{prop}

		The next lemma allows us to connect clusters by short paths if the corresponding vertices in the reduced graph lie in the same component.
		This is the basis of our application of the connected matchings method, a technique that goes back to {\L}uczak~\cite{Luc99} and is by now a standard way of constructing long paths and cycles in dense graphs. See \Cref{appendix:proofs-prelims} for a proof. 

		\begin{lem}[Connecting Paths] 
			\label{lem:luczak} \label{lem:connecting-paths}
			Let $n$ be an integer and let $\eps, d$ be numbers with $0 < 1/n
			\ll \eps \ll d \le 1$.  Let $G=(V,E)$ be an $r$-edge-coloured graph on $n$ vertices
			with a partition $\{V_0, \dots, V_m\}$ and an $(\eps,d)$-reduced graph $\G$ obtained from
			\Cref{lem:regularity}.  Suppose that $W \subset V$ is a
			vertex set such that $|W\cap V_i|\le (d/2)\cdot|V_i|$ for every $i\in[m]$.
			Let $x_i x_j, x_{i'}x_{j'} \in E(\G)$ be two edges in a component of colour $c$.
					
			Then for any two vertices $v \in V_i$ and $w\in V_{j'}$ of typical degree in colour $c$ in $(V_i,V_j)$ and $(V_{i'},V_{j'})$, $G$ contains a $c$-coloured $v$-$w$-path $P$ of order at most $2m$ that avoids all vertices of $W$.
		\end{lem}
		
	\subsection{$b$-matchings}
	
		Adaptations of {\L}uczak's connected matchings method usually proceed by applying the regularity lemma, and finding large matchings in the components of the reduced graph (a matching whose edges belong to the same component is called a \emph{connected matching}). The point is that a connected matching in the reduced graph can easily be converted into a cycle in the original graph.
		
		In our case, it will be more convenient to work with 2-matchings, i.e.\ subgraphs, where each vertex can be incident to at most two edges. These convert to cycles the same way as matchings.
		
		\begin{defn}[Perfect $b$-matching]
			Let $b:V(G)\to \Zp$ be a function on the vertices of a graph $G$. A \emph{perfect $b$-matching} of $G$ is a non-negative function $\om:E(G)\to \Zp$ on the edges, such that $\sum_{w \in N(v)} \om(vw) = b(v)$ for every vertex $v$. When $b$ is the constant 2 function, we call $\om$ a \emph{perfect 2-matching}.
		\end{defn}

		It is easy to see that perfect 2-matchings correspond to vertex-disjoint cycles and edges that cover all the vertices. 
		For example, a perfect matching with weight 2 on each edge is a perfect 2-matching.
		The following analogue of Tutte's theorem is a convenient characterization of graphs that admit a perfect 2-matching (see \cite[Corollary 30.1a]{Sch03}).
		\begin{thm}[Tutte] \label{thm:tutte-2-matching}
			A graph $G$ has a perfect 2-matching if and only if every independent 
			set $S \subset V(G)$ satisfies $|N(S)| \ge |S|$.
		\end{thm}
		
		However, we will need stronger conditions so that our graph is guaranteed to have a perfect 2-matching even after slight modifications. 
		
	\subsection{Robustly matchable graphs}
		
		\begin{defn}[$(\mu,\nu)$-\robmat graphs]
			A graph $H$ on $n$ vertices is $(\mu,\nu)$-\robmat if any of the following two conditions holds.
			\begin{enumerate}[\upshape (1)]
			 \item $\delta(H) \ge (1/2-\mu)n$ and every set of $(1/2-\nu)n$ vertices spans at least $\nu n^2$ edges.
			 \item $H$ is a balanced bipartite graph with parts $A,B$ (of size $n/2$) such that\\
			 - $\delta(H) \ge (1/32-\mu)n$, and\\
			 - all but at most $(1/64+\mu)n$ vertices in $H$ have degree at least $(1/3-\mu)n$.
			\end{enumerate}
			We will distinguish \robmat graphs of type 1 and type 2 accordingly.
		\end{defn}
		
		Note that every $(\mu',\nu')$-\robmat graph with $\mu'\le \mu$ and $\nu'\ge \nu$ is automatically $(\mu,\nu)$-\robmat, as well.
				
		The following claim explains why we call these graphs ``2-matchable''.
		\begin{lem} \label{lem:2match}
			Every $(\mu,\nu)$-\robmat graph $H$ with $\mu\le \nu< {1}/{1000}$ contains a perfect 2-matching.\footnote{Strictly speaking, it would be enough to bound $\mu$ from above by 1/1000: monotonicity implies that $H$ is also $(\mu,\mu)$-\robmat. We impose the upper bound on $\nu$ purely to keep our parameter hierarchy simpler.}
		\end{lem}
		\begin{proof}
			If $H$ is a type 1 \robmat graph, then for every non-empty independent set $S$, we have
			\[ |S| \le  (1/2-\nu)n \le (1/2-\mu)n \le \delta(H) \le |N(S)| \]
			(using the independence of $S$ in the first and last step), so $H$ satisfies the conditions of \Cref{thm:tutte-2-matching}.
			
			Now suppose $H$ is of the second type with bipartition $V(H)=A\cup B$. By K\H{o}nig's theorem, it is enough to check that every independent set has size at most $n/2$. Indeed, this would guarantee the existence of a perfect matching, and hence a perfect 2-matching. So let $S$ be an independent set in $H$, and let $S_A=S\cap A$ and $S_B=S\cap B$. We may assume that $|S_A|\le |S_B|$, and note that $N(S_A)\subset B \setminus S_B$.
			
			If $1\le |S_A|\le ({1}/{64}+\mu)n$, then $|N(S_A)|\ge \delta(H)\ge ({1}/{32}-\mu)n \ge ({1}/{64}+\mu)n\ge |S_A|$. 
			(If $S_A$ is empty, then trivially $|N(S_A)|\ge |S_A|$.)
			So $|S|=|S_A|+|S_B|\le |N(S_A)| + |S_B|\le |B|=  n/2$. On the other hand, if $|S_A|>({1}/{64}+\mu)n$, then there is a vertex $v\in S_A$ of degree at least $(1/3-\mu)n\ge  n/4$, so $|N(S_A)|\ge  n/4$. As $|S_B|\ge |S_A|> (1/64 + \mu)n$, we similarly get $|N(S_B)|\ge  n/4$. But then $|S|=|S_A|+|S_B|\le n- |N(S_A)|-|N(S_B)|\le  n/2$, as needed. 
		\end{proof}
		
		The next two statements illustrate the robustness of the above definition.
		
		\begin{lem} \label{cla:robsub}
			Suppose $H$ is a $(\mu,\nu)$-\robmat graph on $n$ vertices and let $\eps>0$. 
			Suppose $H'$ is a spanning subgraph of $H$ such that $\deg_{H'}(v)\ge \deg_H(v)-\eps n$ for every vertex $v$. 	
			Then $H'$ is $(\mu+\eps,\nu-\eps)$-\robmat whose type coincides with that of $H$.
		\end{lem}
		\begin{proof}
			If $H$ is a type 1 \robmat graph, then $\delta(H')\ge \delta(H)-\eps n\ge (1/2-(\mu+\eps))n$, as needed. Also, $H'$ loses at most $\eps n^2$ edges compared to $H$, so every set of $(1/2-\nu)n$ vertices spans at least $(\nu-\eps) n^2$ edges. In particular, the same holds for every set of $(1/2-(\nu-\eps))n$ vertices.
		
			On the other hand, if $H$ is of the second type, then we similarly get $\delta(H')\ge (1/32-(\mu+\eps))n$, as well as $\deg_{H'}(v)\ge (1/3-(\mu+\eps))n$ for all but $(1/64+\mu)n$ vertices $v$.
		\end{proof}
		
		\begin{lem} \label{cla:redrobust}
			Suppose $H$ is an $r$-edge-coloured $(\mu,\nu)$-\robmat graph on $n$ vertices.
			Let $\HH$ be the $(\eps,d)$-reduced graph of $H$ obtained from \Cref{lem:regularity} with some parameters $\eps,d>0$ and $\ell=2$ (and the corresponding bipartition if $H$ is of type~2).
			
			Then $\HH$ is $(\mu+rd+2\eps,\nu-rd-2\eps)$-\robmat.
			Moreover, the type of $\HH$ coincides with the type of $H$.
		\end{lem}

		\begin{proof}
			Suppose that $\HH$ has $m$ vertices.
			If $H$ is a \robmat graph of the first type, then \Cref{prop:reducedgraph}\ref{itm:reducedgraph-a} guarantees that $\delta(\HH)\ge (1/2-\mu-rd-\eps)m$, and \Cref{prop:reducedgraph}\ref{itm:reducedgraph-c} implies that every set of $(1/2-\nu)m$ vertices spans at least $(\nu-rd-\eps)m^2$ edges.

			Now suppose that $H$ is of type~2 with bipartition $\{A,B\}$.
			Then $\HH$ is balanced bipartite, as well.
			By~\Cref{prop:reducedgraph}\ref{itm:reducedgraph-a}, we have $\delta(\HH) \ge (1/32-\mu-rd-\eps)m$.
			By~\Cref{prop:reducedgraph}\ref{itm:reducedgraph-b}, we have $\deg_{\HH}(x) \ge (1/3-\mu-rd-\eps)m$ for all but at most $(1/64+\mu+\eps)m$ vertices $x \in V(\HH)$.
		\end{proof}

		We will also need the following lemma, which provides sufficient conditions for the existence of $b$-matchings in a graph. 

		\begin{lem} \label{lem:rob-b-matching}
			Let $t,\gamma$ be constants, and let $H$ be a $(\mu,\nu)$-\robmat graph on $m$ vertices such that ${m}/{t} \le \gamma \le \mu \le {\nu}/{4} < {1}/{4000} $. 
			Then $H$ has a perfect $b$-matching for every function $b : V(H) \to \Zp$ such that 
			\begin{enumerate}[\upshape (a)]
				\item 
					$(1 - \gamma) \bnd \le b(x) \le \bnd$ for every $x\in V(H)$,
				\item 
					$\sum_{x\in \ccompt} b(x)$ is even for every component $\ccompt$ of $H$, and
				\item 
					if $H$ is of type~2 with bipartition $\{X,Y\}$, then $\sum_{x \in X} b(x) = \sum_{y \in Y} b(y)$.
			\end{enumerate}
		\end{lem}
		
		\begin{proof}
			As $\sum_{x\in \ccompt} b(x)$ is even in every component $\ccompt$, we can pair up the vertices with odd $b(x)$ within each component. 
			Consider a family $\PP$ that contains one path in $H$ between each such pair
			and let $\om_0 : E(H) \to \Zp$ be the function for which $\om_0(e)$ is the number of paths in $\PP$ containing $e$. Then it is easy to see that $b_0(x) = \sum_{y \in N(x)} \om_0(xy)$ is odd if and only if $b(x)$ is odd, so $b_1(x)=b(x)-b_0(x)$ is even for every $x$. Then for every vertex $x$,
			\[ (1-2\gamma)\bnd \le (1-\gamma)\bnd - m \le b_1(x) \le \bnd, \]
			and if $H$ is of type~2, then we also have 
			\[ 
				\sum_{x \in X} b_1(x) = \sum_{x \in X} b(x) - \sum_{e\in E(H)}\om_0(e) = \sum_{y \in Y} b(y) - \sum_{e\in E(H)}\om_0(e) = \sum_{y \in Y} b_1(y). 
			\]
				
			Let $H'$ denote the graph obtained by replacing each vertex $x$ by a set $W(x)$ of size $b_1(x)/2$ and replacing each edge $xy$ by a complete bipartite graph with bipartition $W(x) \cup W(y)$. Then $H'$ has $n=\sum_{x\in V(H)} b_1(x)/2$ vertices, and $(1-2\gamma)\bnd m/2 \le n\le \bnd m/2$. We will show that $H'$ has a perfect 2-matching $\om'$. Then $\om_1(xy) = \sum_{x'\in W(x), y'\in W(y)} \om'(x'y')$ is a perfect $b_1$-matching in $H$, and hence $\om(xy) = \om_0(xy) +\om_1(xy)$ is a perfect $b$-matching in $H$.

			\medskip
			Let us first consider the case when $H$ is a \robmat graph of type 1. As $\delta(H)\ge (1/2-\mu)m$, we readily get $\delta(H')\ge (1/2-\mu)(1-2\gamma)m\bnd/2\ge (1/2-\mu-\gamma)n$. As in the proof of \Cref{lem:2match}, it is enough to show that every independent set in $H'$ has size at most $(1/2-\mu-\gamma)n$, because then $|N_{H'}(S)|\ge |S|$ holds for every independent $S$, and we can apply \Cref{thm:tutte-2-matching} to get a perfect 2-matching.
			
			So take any independent set $S$ in $H'$, and observe that if $u\in S\cap W(x)$ and $v\in S\cap W(y)$ for some $x,y$ in $H$, then $xy$ is not an edge of $H$ (otherwise $v$ and $w$ are adjacent in $H'$). So $S\subset \bigcup_{x \in U}W(x)$ for some independent set $U$ in $H$. Since $H$ is of type 1, we have $|U|\le (1/2-\nu)m$.
			Thus
			\[ |S|\le (1/2-\nu)m\bnd/2 \le \frac{1/2-\nu}{1-2\gamma} n \le (1/2-\nu+2\gamma)n \le (1/2-\mu-\gamma)n, \]
			as needed. 
			
			Now suppose that $H$ is of type 2. In this case, $\sum_{x \in X} b_1(x) = \sum_{y \in Y} b_1(y)$ guarantees that $H'$ is also balanced bipartite. Also, $\delta(H)\ge (1/32-\mu)m$ implies $\delta(H')\ge (1/32-\mu-\gamma)n$ as before. Moreover, for every vertex $x$ of degree at least $(1/3-\mu)m$ in $H$, we get that every vertex in $W(x)$ has degree at least $(1/3-\mu-\gamma)n$ in $H'$. As there are at most $(1/64+\mu)\bnd m/2<(1/64+\mu+\gamma)n$ exceptions, we see that $H'$ is $(\mu+\gamma,\nu)$-\robmat. In particular, by \Cref{lem:2match}, it has a perfect 2-matching.
		\end{proof}

	\subsection{Cycle covers in unbalanced bipartite graphs}
	
		Another tool we need is the following variant of a lemma of Erd\H{o}s, Gy\'arf\'as and Pyber~\cite{EGP91}. It finds a monochromatic cycle cover of the smaller part of an unbalanced bipartite graph if this part has large minimum degree.

		\begin{lem}[Erd\H{o}s--Gy\'arf\'as--Pyber~\cite{EGP91}]\label{lem:EGP}
			Let $H$ be an $r$-coloured bipartite graph with bipartition
			$\{A,B\}$. Suppose that $|A| \ge 100^3 r^3 |B|$ and that
			every vertex in $B$ has at least $|A|/100$ neighbours in $A$.
			Then there are $100r^2$ monochromatic pairwise vertex-disjoint
			proper cycles and edges that together cover all vertices of $B$.
		\end{lem}
		
		We give a short proof of this lemma for completeness in \Cref{appendix:proofs-prelims}, although our argument is nearly identical to the original proof. 
		
	\subsection{Random sampling}

		The following lemma is a well-known Chernoff-type bound on the tail of the binomial distribution (see e.g. \cite[Theorem 2.1]{JLR01}).

		\begin{lem}[Chernoff bound]\label{lem:che}
			Let $X \sim \Bin(n,p)$ be a binomial random variable. Then the following bounds hold for every $0 \le a \le 1$.
			\begin{itemize}
				\item $\Pr\left[X<(1-a)np\right] \le e^{-a^2np/2}$, and
				\item $\Pr\left[X>(1+a)np\right] \le e^{-a^2np/3}$.
			\end{itemize}
		\end{lem}

		In our proof, we will need a small set of vertices that contains many neighbours of every large-degree vertex. As shown by the next result, a randomly chosen set satisfies these properties. The proof is a routine application of Chernoff-type bounds; see \Cref{appendix:proofs-prelims}.
		
		\begin{prop} \label{prop:sample}
			Let $G$ be a graph on $n$ vertices with $(\eps,d)$-regular partition $\{V_0,\dots,V_m\}$ as provided by \Cref{lem:regularity}. Also, let $p$ be a positive parameter, and let $B\subset V=V(G)$ be a vertex set satisfying $V_0\subset B$ and $|B\cap V_i|\le 10p|V_i|$ for every $i\in[m]$. If ${m\log n}/{\sqrt{n}}<p< {1}/{100}$ and $\eps< {1}/{10}$, then there is a set $A\subset V\sm B$ with the following properties.
			\begin{enumerate}[\upshape (a)]
				\item $|A|\ge  (p/2)n$,
				\item \label{itm:sample-b}
					$|A\cap V_i|\le 2p|V_i|$ for every $i\in[m]$,
				\item $\deg(v,A\cap V_i)\ge (p/2)\deg(v,V_i)$ for every $v\in V$ and $i\in[m]$ with $\deg(v,V_i)>30p|V_i|$,
				\item $\deg(v,A)\ge  {|A|}/100$ for every vertex $v\in V$ with $\deg(v,V\sm B)> {n}/{40}$.
			\end{enumerate}
		\end{prop}

\section{Main proof} \label{sec:main}

	The proof of \Cref{thm:main} comes as a combination of the following two results.
	
	\begin{restatable}{thm}{structlem} \label{thm:struct}
		Let $1/n \ll \mu \ll 1$, and let $G$ be an $r$-edge-coloured graph on $n$ vertices with minimum degree $\delta(G) \ge n/2 + 1200 r \log n$. 
		Then the vertices of $G$ can be partitioned into at most $\qinit$ monochromatic cycles and a $(\mu,20\mu)$-\robmat graph $H$ on at least $n/2$ vertices.
	\end{restatable}

	\begin{thm} \label{thm:robmatpart}
		Let $1/n\ll \mu\le  {\nu}/{20}\ll 1$. Every $r$-edge-coloured $(\mu,\nu)$-\robmat graph on $n$ vertices can be partitioned into $(1/\mu + 200)r^2$ monochromatic cycles.
	\end{thm}
	
	Let us first elaborate on the conditions that are implicit in our $\ll$ notation. We will need to select the five parameters $\nu,\mu,d,\eps$ and $n$, in this order. As a point of reference, we describe here the exact constraints that come from the proofs of \Cref{thm:struct,thm:robmatpart}:
	\begin{equation} \label{eq:params}
	\begin{split}
		\nu &< \frac{1}{1000}, \\
		\mu &< \min \left\{ \frac{1}{700000} , \frac{\nu}{20}  \right\}, \\
		d &\le \frac{\mu}{r}, \\
		\eps &< \min \left\{ \frac{1}{10^{13} r^6} , \frac{\mu^4}{20^4},  \frac{d^2}{4000},  \eps_{\ref{lem:inner-paths}}(d), \eps_{\ref{lem:connecting-paths}}(d)  \right\}, \\
		n &> \max \left\{ \frac{4}{\eps}(M_{\ref{lem:regularity}}(\eps,r,2))^4, n_{\ref{lem:inner-paths}}(\eps), n_{\ref{lem:connecting-paths}}(\eps)  \right\},
	\end{split}
		\tag{$\ll$}
	\end{equation}
	where $M_{\ref{lem:regularity}}, \eps_{\ref{lem:inner-paths}}, \eps_{\ref{lem:connecting-paths}}, n_{\ref{lem:inner-paths}}$ and $n_{\ref{lem:connecting-paths}}$ are the appropriate constants coming from \Cref{lem:inner-paths,lem:regularity,lem:connecting-paths}. Let us emphasize that ${1}/{n} < \eps < d < \mu < \nu < 1$.
	
	\medskip
	The proof of \Cref{thm:struct} is a somewhat technical argument that shows that either $G$ is already \robmat of type 1, or it can be turned into a type 2 robustly $2$-matchable graph by deleting few monochromatic cycles. We defer its proof to \Cref{sec:lemmas}, and proceed with the proof of \Cref{thm:robmatpart}.
	
	\begin{proof}[Proof of \Cref{thm:robmatpart}]
		Let $G=(V,E)$ be an $r$-edge-coloured $(\mu,\nu)$-\robmat graph with $20\mu\le \nu$.
		If $G$ is of type~2, then it is a balanced bipartite graph, and we denote its bipartition by $\{A,B\}$.
		By~\eqref{eq:params}, we are guaranteed a partition $V_0, V_1,\dots,V_m $ of $V(G)$ as detailed in \Cref{lem:regularity}.
		Let $\G$ be the corresponding $(\eps,d)$-reduced graph.
		If $G$ is of type~2, then $\G$ is also balanced bipartite, and we denote its bipartition by $\{\calA,\calB\}$.
		Note that $\G$ has $m\le M_{\ref{lem:regularity}}(\eps,r,2)$ vertices.
		
		By \Cref{cla:redrobust}, and using $\eps\le \mu/2$ and $d\le \mu/r$, we know that $\G$ is $(3\mu,\nu-2\mu)$-\robmat.
		Let $\HH$ denote the subgraph of $\G$ that consists of all edges contained in monochromatic components of order at least $(\mu/r)m$. 
		Then $\HH$ is the union of at most $\qcompt$ monochromatic components, and $\deg_{\HH}(x)\ge \deg_{\G}(x)-\mu m$ for every vertex $x$ in $\G $.
		By \Cref{cla:robsub}, $\HH$ is $(4\mu, \nu-3\mu)$-\robmat.
		Moreover, the type of $\HH$ coincides with the type of $G$.
		As $20\mu \le \nu \le {1}/{1000}$ (by \eqref{eq:params}), \Cref{lem:2match} implies that $\HH$ contains a perfect 2-matching $\om$.
		We denote by $\M$ the edges of $\HH$ that have non-zero weight under $\om$.
		
		\medskip
		
		Let us now call a vertex $v\in V_i$ ($i\in [m]$) \emph{good} if $v$ has typical degree in each regular pair $(V_i,V_j)$ that corresponds to an edge of $\M$. In other words, $v$ is good if $\deg_c(v,V_j)\ge (d-\eps)|V_j|$ for each edge $x_ix_j\in\M$ of colour $c$. We call all other vertices of $G$ \emph{bad}.
		
		\begin{claim}\label{cla:bad-vertices}
			There is a collection $\Cbad$ of at most $\qbad$ vertex-disjoint monochromatic proper cycles and edges in $G$ covering all bad vertices such that
			\begin{equation}\label{eq:badcoverratio}
				|V_i \cap V(\Cbad)| \le 5\sqrt{\eps} |V_i| \qquad \text{for every $i\in [m]$.}
				\end{equation}	 
		\end{claim}

		\begin{proof}
			Let $B$ be the set of bad vertices (note that $V_0\subset B$). By \eqref{equ:typical-degree-in-reg-pair}, and because $\M$ is a 2-matching, we know that $|B\cap V_i|\le 2\eps |V_i|$ for every $i\in [m]$. 
			In particular, $|B|\le 2\eps|V_i|\cdot m + |V_0|  \le 2\eps|V_i|\cdot m + \eps n \le 3\eps n$.
		
			This together with~\eqref{eq:params} means that we can apply \Cref{prop:sample} with $p=2\sqrt{\eps}$ to obtain a set $A$ of size $|A|\ge \sqrt{\eps} n \overset{\eqref{eq:params}}{\ge} 3\cdot 100^3r^3\eps n \ge 100^3 r^3|B|$ 
			such that $|A\cap V_i|\le 4\sqrt{\eps}|V_i|$ for every $i\in[m]$, and each vertex $v\in G$ with $\deg_G(v,V\sm B)> n/40$ has at least $|A|/100$ neighbours in~$A$. As $\delta(G)\ge (1/32-\mu)n$ and $|B|<3 \eps n$, this actually holds for every vertex of $G$, and in particular for every vertex in $B$. But then \Cref{lem:EGP} provides a set $\Cbad$ of at most $\qbad$ disjoint monochromatic proper cycles and edges covering $B$. Note that the vertices of $\Cbad$ are contained in $A\cup B$, so \eqref{eq:badcoverratio} clearly holds.
		\end{proof}
		
		\medskip
		
		\begin{claim} \label{cla:ch}
			There is a collection $\Ch$ of at most $\qcompt$ vertex-disjoint monochromatic proper cycles and edges in $G$, all disjoint from $\Cbad$, such that
			\begin{enumerate}[\upshape (a)]
				\item \label{itm:auxcycconv_vxs} 
					for every edge $e=x_ix_j$ of $\HH$, there is an edge $u_ev_e$ of colour $c(e)$ in $\Ch$ between vertices $u_e\in V_i$ and $v_e\in V_j$ that have typical degree in the regular pair $(V_i,V_j)$, and
				\item \label{itm:auxcycconv_small} 
					$|V_i \cap V(\Ch)| \le \eps |V_i|$ for every $i\in [m]$.
			\end{enumerate}
		\end{claim}

		\begin{proof}
			We will apply a simple algorithm to find one cycle for each monochromatic component of $\HH$. For this, take a component $\compt$ of colour $c$, and let $e_1,\dots,e_s\in E(\HH)$ be its edges. We perform the following two steps:
			
			\begin{enumerate}[\upshape (1)]
				\item \label{itm:step-1}
					 For $i = 1, \ldots, s$, let $e_i=y_iz_i$, and pick $u_i\in V(y_i)$ and $v_i\in V(z_i)$ that are not yet used, but have typical degree in the regular pair $(V(y_i),V(z_i))$, and $u_iv_i$ is a $c$-coloured edge in $G$.
				 \item \label{itm:step-2}
					 Use \Cref{lem:luczak} to find a $c$-coloured $v_i$-$u_{i+1}$ path $P_i$ in $G$ of order at most $2m$ that avoids all previously used vertices (except $v_i$ and $u_{i+1}$), for $i = 1, \ldots, s$.
			\end{enumerate}
			
			If these steps work, then $C_{\compt}=u_1v_1P_1u_2v_2P_2\dots u_sv_sP_su_1$ is a $c$-coloured cycle that takes care of all edges in $\compt$. Repeating this for every component gives us $\qcompt$ disjoint monochromatic cycles satisfying Condition \ref{itm:auxcycconv_vxs}. Condition \ref{itm:auxcycconv_small} is also satisfied because the edges and paths produced by these steps use at most $|E(\HH)|\cdot 2m\le m^3 \overset{\eqref{eq:params}}{\le} \eps|V_i|$ vertices in $G$ (using $n> {2m^4}/{\eps}$ and $|V_i|> {n}/{(2m)}$ in the second inequality). We just need to check that these steps can indeed be applied.
			
			For Step \ref{itm:step-1}, note that by \eqref{equ:typical-degree-in-reg-pair}, $V(y_i)$ and $V(z_i)$ each have at least $(1-\eps)|V(y_i)|$ typical vertices, of which at most $2\eps |V(y_i)|$ have been used in former steps and at most $5\sqrt{\eps}|V_i|$ are in $\Cbad$ by~\eqref{eq:badcoverratio}, as noted above. But then there is an edge between unused typical vertices in colour $c$ because $\eps<1/100$ and $(V(y_i),V(z_i))$ is $\eps$-regular. For Step \ref{itm:step-2}, we just need to apply \Cref{lem:luczak} with the set $W$ consisting of the vertices of $\Cbad$ in $V(y_i)\cup V(z_i)$, as well as all previously used vertices except $v_i$ and $u_{i+1}$. This is possible because $|W|<12\sqrt{\eps}|V_i|$.
		\end{proof}
		
		Note that $\Cbad$ and $\Ch$ together contain at most $\qrest$ cycles. For parity reasons, we need another small collection $\Cs$ of single vertices.
		For every component $\ccompt$ of (the uncoloured graph)~$\HH$, add a single vertex of $\bigcup_{x\in \ccompt} V(x) \sm V(\Cbad \cup \Ch)$ to $\Cs$ if $\big|\bigcup_{x\in \ccompt} V(x) \sm V(\Cbad \cup \Ch)\big|$ is odd.
		Since $\HH$ is $(4\mu, \nu-3\mu)$-\robmat, $\HH$ has at most two components.
		Thus we have $|\Cs| \leq 2$.
		Moreover, if $\HH$ is of type 2, then it has only one component with an even number of vertices.
			So in this case $\Cs$ is just the empty set.
		Write $\Cn = \Cbad \cup \Ch \cup \Cs$, and note that
		\begin{equation}\label{equ:VCn-bounded}
			|V_i\cap V(\Cn)|\le |V_i\cap V(\Cbad\cup \Ch \cup \Cs)| \leq (5\sqrt{\eps} + \eps)|V_i| +2 \leq 6 \sqrt{\eps} |V_i|
		\end{equation}
		for every $i\in[m]$.
		The rest of the proof will extend the cycles in $\Ch$ so that they cover all the remaining vertices.

		More precisely, $\Ch$ will serve as the ``skeleton'' of our cycle cover in the sense that we will use \Cref{lem:inner-paths} to replace each edge $u_ev_e$ (corresponding to some $e=x_ix_j$ in $\HH$) with a $u_e$-$v_e$ path $P_e$ in $(V_i,V_j)$. But first we need to decide how long these paths should be. So fix an $\ell$ such that
		\begin{equation} \label{eq:elldef}
			(1 - \eps^{1/4})|V_i| \le \ell \le (1 - \eps^{1/4})|V_i| + 2 \text{ and $\ell$ is divisible by 2.}
		\end{equation}
		Our plan is to cover at least $\ell$ vertices in each cluster by the paths corresponding to the edges of the 2-matching $\M$.
		By~\eqref{equ:VCn-bounded}, this leaves $b(x_i)$ vertices in $V_i$, where
		\begin{equation}\label{equ:bx-bound}
			\eps^{1/4}|V_i| \ge b(x_i)=|V_i \sm V(\Cn)|-\ell \geq \eps^{1/4} |V_i| - |V_i \cap \Cn| -2 \geq 0
		\end{equation}
		vertices in each $V_i$. 
		
		\begin{claim} \label{claim:b-matching-H}
			$\HH$ has a perfect $b$-matching $\om_0:E(\HH)\to \Zp$.
		\end{claim}

		\begin{proof}
			By~\eqref{equ:bx-bound} and since $|V_i \cap \Cn|+2 \leq 12\sqrt{\eps}|V_i| +2  \overset{\eqref{eq:params}}{\le} \mu \eps^{1/4}|V_i|$, we have
			\[ 
				 \quad \eps^{1/4}|V_i| \quad \ge \quad b(x_i) \quad \ge \quad \eps^{1/4}|V_i| - |V_i \cap \Cn| -2 \quad \overset{}{\ge} \quad (1-\mu)\eps^{1/4}|V_i|.
			\]
			Moreover, the definition of $\Cs$ implies that $\sum_{x\in \ccompt} b(x)$ is even for every component $\ccompt$ of $\HH$.
			Recall that $\HH$ is $(4\mu, \nu-3\mu)$-\robmat.
			If $G$ is of type~1, then $\HH$ is of type~1 and we can finish by \Cref{lem:rob-b-matching}.

			Suppose that $G$ is of type~2 and thus $\HH$ is of type~2 as well.
			Recall that $G$ has bipartition $\{A,B\}$ and $\HH$ has bipartition $\{\calA,\calB\}$.
			Moreover, $\bigcup_{x\in \calA} V(x)\subset A$ and $\bigcup_{y\in \calB} V(y)\subset B$.
			Since $G$ is balanced bipartite and $\Cbad \cup \Ch$ consists of proper cycles and edges, it follows that $\sum_{x\in \calA}b(x) = \sum_{y\in \calB}b(y)$.
			Therefore, we can finish by \Cref{lem:rob-b-matching}.
			Note that the set $\Cs$ did not play a role in this case (and is anyways empty).
		\end{proof}
		
		Let $\om_0$ be the perfect $b$-matching guaranteed by \Cref{claim:b-matching-H}. 
		Define $\om:E(\HH)\to \Zp$ as
		\[
			\om(x_ix_j) = \begin{cases} 
				\om_0(x_ix_j) & \text{ for $x_ix_j\notin\M$}, \\
				\om_0(x_ix_j) + \frac{\ell}{\deg_{\M}(x_i)} & \text{ for $x_ix_j\in\M$.}
			\end{cases}
		\]
		Note that this is well-defined because $\deg_{\M}(x_i)=\deg_{\M}(x_j)$, and integral because $\ell$ is even and $\deg_{\M}(x_i)\in [2]$. Then for every vertex $x_i$ in $\HH$, we have $\sum_{x_i\in e\in \HH} \om(e) = |V_i \sm V(\Cn)|$ and $\sum_{x_i\in e\in \HH \sm \M} \om(e) \le b(x_i)\le \eps^{1/4} |V_i|$.
		
		\begin{claim}
			For every edge $e=x_ix_j$ in $E(\HH)$, there is a $u_e$-$v_e$ path $P_e$ of colour $c(e)$ in $G[V_i,V_j]$ that contains exactly $\om(e)+1$ vertices in each of $V_i$ and $V_j$. Moreover, these paths can be chosen so that they are internally vertex-disjoint from each other and from $\Cn$.
		\end{claim}

		\begin{proof}
			Let us first apply \Cref{prop:sample} with $p=2\sqrt{\eps}$ and $B=V(\Cn)$ to get a set $S^1$ with the properties given in the statement of the proposition, and then apply it again with the same $p$ and $B=V(\Cn)\cup S^1$ to get another such set $S^2$. This is possible because $V_0\subset V(\Cn)$ holds, and we also have $|V(\Cn)\cap V_i|\le 6\sqrt{\eps}|V_i|$, and thus $|S^1\cap V_i| \le 4\sqrt{\eps}|V_i|$ for every $i\in[m]$. Let $S^b_i = S^b\cap V_i$ for every $i\in[m]$ and $b\in[2]$. Then
			\begin{enumerate}[\upshape (a)]
				\item 
					$|S^b_i| \le 4\sqrt{\eps}|V_i|$ for every $i\in[m]$ and $b\in[2]$, and
				\item \label{itm:typical-deg} 
					for every edge $x_ix_j$ in $\HH$ of colour $c$ and every vertex $v\in V_j$ with typical degree in the regular pair $(V_i,V_j)$, we have $\deg_{c}(v,S^b_i)\ge 6\eps |V_i|$ for $b\in[2]$. 
			\end{enumerate}
			To see \ref{itm:typical-deg}, note that every such vertex $v$ of typical degree satisfies $\deg_c(v,V_i)\ge (d-\eps)|V_i| > 60\sqrt{\eps}|V_i|$ (using $d \overset{\eqref{eq:params}}{>} 61\sqrt{\eps}$), so by \Cref{prop:sample}, $\deg_{c}(v,S^b_i)\ge \sqrt{\eps}(d-\eps)|V_i| \overset{\eqref{eq:params}}{>} 60\eps|V_i|$. 
			
			\medskip
			
			Let us first consider the edges $e_1,\dots,e_s$ of $\HH \sm \M$. We will find the $u_k$-$v_k$ paths $P_k$ (where $u_k v_k$ is the edge in $\Ch$ corresponding to $e_k$, as obtained in \Cref{cla:ch}) one by one so that for every $k$, the vertex set of $\PP_k= \bigcup_{j=1}^{k-1} P_j$ is disjoint from each $S^2_i$, and intersects each $S^1_i$ in at most $k-1$ vertices. Suppose we have already found $P_1,\dots,P_{k-1}$. Let us also assume that $u_k\in V_i$ and $v_k\in V_j$ (so $e_k=x_ix_j$), and let $c$ be the colour of $e_k$.
			
			If $\om(e_k)=0$, then there is nothing to do: we can take $P_k=u_kv_k$. If $\om(e_k)=1$, then $\deg_c(u_k,S^1_j\sm V(\PP_k)) \ge 4\eps|V_j|-k \overset{\eqref{eq:params}}{\ge} \eps |V_j|$ (using $\eps|V_i|> \eps n/(2m) > m^2$) and similarly, $\deg_c(v_k,S^1_i\sm V(\PP_k)) \ge \eps|V_i|$. Hence, by regularity, we can find adjacent vertices $u\in S^1_i\sm V(\PP_k)$ and $v \in S^1_j\sm V(\PP_k)$ such that $P_k=u_kvuv_k$ is a $c$-coloured path, as needed.
			
			So suppose $\om(e_k)>1$. Let $W=V(\Cn)\cup S^1\cup S^2 \cup V(\PP_k)$ be the set of ``forbidden'' vertices. We will again need neighbours $u\in S^1_i\sm V(\PP_k)$ and $v\in S^1_j\sm V(\PP_k)$ of $v_k$ and $u_k$ respectively, but this time we want to apply \Cref{lem:inner-paths} to connect them with a $u$-$v$ path of the right length that avoids $W$.
			
			We have seen above that $\deg_c(u_k,S^1_j\sm V(\PP_k)) \ge \eps |V_j|$. Also,
			\[ 
				|V_i \sm W| \ge |V_i| - |V_i\cap V(\Cn)| - |S^1_i| - |S^2_i| - b(x_i) \ge  (1 - 14\sqrt{\eps} - \eps^{1/4})|V_i| \overset{\eqref{eq:params}}{\ge} |V_i|/2, 
			\] 
			so by regularity, there is a neighbour $v\in S^1_j\sm V(\PP_k)$ of $u_k$ such that $\deg_c(v,V_i\sm W)\ge (d-\eps)|V_i|/2 \overset{\eqref{eq:params}}{\ge} \eps |V_i|$. 
			Similarly, there is a neighbour $u\in S^1_i\sm V(\PP_k)$ of $v_k$ such that $\deg_c(u,V_j\sm W)\ge \eps |V_i|$. As $\om(e_k) \le \eps^{1/4}|V_i| \overset{\eqref{eq:params}}{\le} (1-\sqrt{\eps})|V_i\sm W|$, 
			we can apply \Cref{lem:inner-paths} (with $U_1 = V_i \sm W$ and $U_2 = V_j \sm W$) to find a $c$-coloured $v$-$u$ path $P'$ of order $2\om(e_k)$ that is internally vertex-disjoint from $W$. But then $P_k=u_kvP'uv_k$ is a path satisfying our requirements.
			
			\medskip
			
			Finally, let $e_{s+1},\dots,e_{s+t}$ be the edges of $\M$. Note that each vertex of $\HH$ is incident to exactly one or two of these edges. Using the same notation as before, we will find the $u_k$-$v_k$ paths $P_k$ so that $\PP_k$ is disjoint from $S^2_i$ unless $e_k=x_i x_j$ is the last edge at $x_i$ (according to the ordering), and similarly for $x_j$. 
			
			Fix $k$, and let $U_i=V_i\sm (V(\Cn)\cup V(\PP_k))$ if $e_k$ is the last edge at $x_i$, and let $U_i=V_i\sm (V(\Cn)\cup V(\PP_k)\cup S^2_i)$ otherwise. 
			Using $|S^2_i|\le \ell/2$ and the assumption that $\eps$ is small (see \eqref{eq:params}), it is easy to check from the definitions that we have
			$  |U_i| \ge \om(e_k) \ge \ell/2\ge |V_i|/3$.  
			We similarly get $|U_j|\ge  \om(e_k) \ge |V_i|/3$ for the analogously defined $U_j$. 
			
			We want to use \Cref{lem:inner-paths} to find the required $u_k$-$v_k$ path $P_k$ of order $2(\om(e_k)+1)$. As $\min\{|U_i\cup\{u_k\}|,|U_j\cup\{v_k\}|\}\ge \om(e_k)+1$, we just need to check that $\delta(G[U_i,U_j]) \ge 5\eps|V_i|$. This follows from the properties of $S^1_i$ and $S^2_i$: If $e_k$ is the last edge at $x_i$, then $S^2_i\subset U_i$, and otherwise all but $k$ vertices of $S^1_i$ are in $U_i$. Either way, by~\ref{itm:typical-deg} we obtain $\deg_c(v,U_i)\ge 6{\eps}|V_i|-k \ge 5\eps|V_i|$ for every $v\in U_j$, and similarly, $\deg_c(u,U_j)\ge 5\eps|V_j|$ for every $u\in U_i$, as needed.
			(Here we also used that the vertices of non-typical degree are all in $V(\Cn)$ by \Cref{cla:bad-vertices}.)
		\end{proof}
		
		Now for every $e\in E(\HH)$, we replace the edge $u_ev_e$ with the path $P_e$ in the appropriate cycle of~$\Ch$. This gives us $\qrest+2 \leq (1/\mu + 200)r^2$ monochromatic cycles that cover all vertices in $V_0$, and (by the definition of the function $\om$) $|V_i|$ vertices in each $V_i$. In other words, we find a monochromatic cycle partition of $G$, as needed.
	\end{proof}
\section{The structural lemma}\label{sec:lemmas}

	Let us now prove the main structural theorem from \Cref{sec:main}.

	Our proof makes use of a classical result of Bondy and Simonovits on the extremal number of even cycles.

	\begin{thm}[Bondy--Simonovits, \cite{BS74}] \label{thm:exact-length-cycles}
		Let $G$ be a graph on $n$ vertices with at least $e$ edges, and let $\ell$ be such that $\ell \le e/(100n)$ and $\ell n^{1/\ell} \le e/(10n)$. Then $G$ contains a cycle of length exactly $2\ell$.
	\end{thm}

	More precisely, we need the following immediate corollary:
	
	\begin{cor} \label{cor:exact-length-cycles}
		Let $\ell \ge  (2 / \log 10) \log n$ be even. Then every graph on $n$ vertices with average degree at least $100\ell$ contains a cycle of length $\ell$.
	\end{cor}
  
	The main idea of the proof of \Cref{thm:struct} is to show that if $G$ is not \robmat of type 1, then it has a bipartition $\{X,Y\}$ such that $G[X,Y]$ is essentially \robmat of type 2, except it might be unbalanced. We use \Cref{cor:exact-length-cycles} to balance out this bipartite subgraph by covering some vertices of $G$ with cycles induced by $X$ and $Y$.
	  
  	\begin{proof}[Proof of \Cref{thm:struct}]
		Let us assume that $G$ is not $(\mu,20\mu)$-\robmat of type 1. The following claim provides us with useful information regarding the structure of $G$. Its proof, while somewhat technical, is routine.
  
		\begin{claim}
			There is a partition $\{X, Y\}$ of the vertices with the following properties.
			\begin{enumerate}[\upshape (a)]
				\item \label{itm:X-Y-size} 
					$|X| \ge |Y| \ge \frac{n}{2} - 5400\mu n$,
				\item \label{itm:min-deg} 
					$\delta(G[X, Y]) \ge \frac{n}{16} - 10800\mu n$,
				\item \label{itm:almost-min-deg} 
					all but $10800\mu n$ vertices have degree at least $\frac{2n}{5}-10800\mu n$ in $G[X,Y]$, and
				\item \label{itm:X-max-deg} 
					if $|X| > \frac{n}{2}$ we have $\Delta(G[X]) \le \frac{n}{16}$.
			\end{enumerate}        
		\end{claim}

  		\begin{proof}
  			As $\delta(G)\ge {n}/{2}$, the assumption that $G$ is not type~1 $(\mu,20\mu)$-\robmat implies the existence of a set $S_0$ of size at least $({1}/{2} - 20\mu)n$ that spans fewer than $20\mu n^2$ edges. This set $S_0$ cannot contain more than $480\mu n$ vertices $v$ satisfying $\deg_G(v,S_0)\ge {n}/{12}$, so there is a subset $S_1\subset S_0$ of size exactly $({1}/{2}-500\mu)n$ such that $\Delta(G[S_1])\le {n}/{12}$. 
  
  			Now let $T$ be the set of vertices not in $S_1$ that send at least ${2n}/{5}$ edges into $S_1$, and let $S_2$ be the set of vertices not in $S_1\cup T$. If $q$ denotes the size of $S_2$, then we have $|S_1|={n}/{2}-500\mu n$ and $|T|={n}/{2}+500\mu n-q$. We can bound $q$ by double-counting the edges of $G$ between $S_1$ and $T\cup S_2$. Indeed, counting from $S_1$, the number of such edges is at least $|S_1| \cdot {n}/{2}-40\mu n^2= ({n}/{2})^2-290\mu n^2$. On the other hand, counting from $T\cup S_2$, there are at most 
			\[ 
				|T||S_1|+|S_2|\cdot \frac{2n}{5} \le \left(\frac{n}{2}+500\mu n-q\right) \frac{n}{2} +q\cdot \frac{2n}{5} = \left(\frac{n}{2}\right)^2 + 250\mu n^2 - q\cdot \frac{n}{10} 
			\]
			such edges. Putting these together, we get that $q\cdot {n}/{10}\le 540\mu n^2$, so $q \le 5400 \mu n$. 
  
  			Setting $S=S_1\cup S_2$, we obtain the following bounds on the degrees in $G[S,T]$. 
			\begin{align} \label{eq:degreecond}
			\begin{split}
				\text{For every $v\in T$, } ~\quad \deg(v,S) &\ge \frac{2n}{5};   \\
				\text{for every $v\in S_1$, } \quad \deg(v,T) &\ge \frac{n}{2}-\Delta(G[S_1])-|S_2| \ge \frac{2n}{5}-5400\mu n;   \\
				\text{for every $v\in S_2$, } \quad \deg(v,T) &\ge \frac{n}{2}-\frac{2n}{5}-|S_2| >\frac{n}{16} - 5400\mu n  .
			\end{split}
			\end{align}

  			Now let $X_0$ be the larger of the two sets $S$ and $T$, and let $Y_0$ the smaller one. We then have $|X_0|={n}/{2}+k$ and $|Y_0|={n}/{2}-k$, where $k=|500\mu n-q|\le 5400 \mu n$. 
  			Let $Z\subset X_0$ be the set of vertices in $X_0$ with at least ${n}/{16}$ neighbours in $X_0$.
  
  			If $|Z|\ge k$, then let $Z_0\subset Z$ be a subset of size $k$. We claim that $X=X_0\sm Z_0$ and $Y=Y_0\cup Z_0$ satisfy the conditions. Indeed, as $|Z_0|\le 5400\mu n$, we get $\deg(v,X)\ge \deg(v,X_0)-5400\mu n$ for every vertex $v$ in $Y$, and $\deg(v,Y)\ge \deg(v,Y_0)$ for every $v\in X$. Combining this with \eqref{eq:degreecond}, we see that $\delta(G[X,Y])\ge {n}/{16}-10800\mu n$, and every vertex not in $S_2\cup Z_0$ has degree at least ${2n}/{5}-10800\mu n$, establishing \ref{itm:min-deg} and \ref{itm:almost-min-deg}. As $|X|=|Y|={n}/{2}$, \ref{itm:X-Y-size} and \ref{itm:X-max-deg} are also satisfied.
  
  			If $|Z|<k$, then we take $Z_0=Z$ instead. The same argument shows that \ref{itm:min-deg} and \ref{itm:almost-min-deg} hold. The definition of $Z$ implies $\Delta(G[X])<{n}/{16}$, establishing  \ref{itm:X-max-deg}, while \ref{itm:X-Y-size} holds because $k\le 5400 \mu n$.
		\end{proof}
  
		Let $X$ and $Y$ be as in the claim. If $|X| = |Y|$, then set $H=G[X,Y]$. As we will see, this satisfies the conditions. In the meantime, we may assume that $|X| = \lceil {n}/{2} \rceil + k$ for some $0 \le k\le 5400\mu n$. 

		\medskip
		Let us first consider the case when $k \ge 400 r \log n$. We can write $2k=\ell_1 + \dots + \ell_t$ as the sum of $t \le 400r+1$ even numbers such that ${k}/{(400r)} \le \ell_i \le {k}/{(200r)}$ for every $i$. We will find pairwise disjoint monochromatic cycles $C_1, \dots, C_t$ in $X$, where each $C_i$ has length $\ell_i$. 

		Suppose we have already found $C_1, \dots, C_{i-1}$ with these properties. We want to apply \Cref{cor:exact-length-cycles} to find $C_i$. As $\delta(G)\ge {n}/{2}$, the minimum degree of $G[X]$ is at least $k$, therefore $X$ induces at least ${k|X|}/{2} \ge {kn}/{4}$ edges. On the other hand, \ref{itm:X-max-deg} implies that the vertices of $C_1, \dots, C_{i-1}$ are incident to at most $2k \cdot {n}/{16} = {kn}/{8}$ edges. That is, at least half of the edges in $G[X]$ are not incident to any of the cycles $C_1, \dots, C_{i-1}$, and hence the average degree induced by $X' = X \sm V(C_1 \cup \dots \cup C_{i-1})$ is at least ${k}/{2}$. The average degree of $G[X']$ in the most common colour (say blue) is then at least ${k}/{(2r)} \ge 100\ell_i$, so \Cref{cor:exact-length-cycles} provides a blue cycle of length $\ell_i$, as needed (using $\ell_i \ge \log n \ge (2 / \log 10) \log n$).

		Let $\C$ be the set of cycles $C_1, \dots, C_t$.
		If $n$ is odd, we add another vertex in $X$ as a singleton cycle to $\C$.
		Then $\C$ contains at most $400r+2$ cycles, and $A = X \sm V(\C)$ and $B = Y \sm V(\C)$ satisfy $|A| = |B| = {n}/{2} - k \ge ({1}/{2} - 5400\mu)n$. In this case, we choose $H = G[A, B]$. 

		\medskip
		Finally, suppose $0 \le k < 400r \log n$, and let $\log n \le \ell < \log n + 2$ be even. We can write $\ell + 2k = \ell_1 + \dots + \ell_t$ as the sum of $t \le 200r + 1$ even numbers such that $\log n \le \ell_i \le 2\log n$ for every $i$. We will again find a monochromatic cycle $C_i$ of length $\ell_i$ in $G[X]$ for every $i$, but this time we will also need an $\ell$-cycle $C$ induced by $Y$ to balance out the graph.

		The minimum degree of $G[Y]$ is at least ${n}/{2}+ 1200 r \log n - ({n}/{2}+ k) \ge 100 r \ell$, so the most common colour (say blue) has average degree at least $100 \ell$. By \Cref{cor:exact-length-cycles}, there is a blue cycle $C$ of length $\ell$ in $Y$.

		To find $C_1, \dots, C_t$, we use the same argument as before. Suppose we already have $C_1, \dots, C_{i-1}$. The minimum degree of $G[X]$ is at least $1200 r\log n$, so $X$ induces at least $1200r |X|\log n /2 \geq {300}r n\log n$ edges. Out of these, at most $2k \cdot {n}/{16} ={kn}/{8} \le 50r n \log n$ are incident to some of the cycles $C_1, \dots, C_{i-1}$. Hence, the average degree in $G[X']$, where $X' = X \setminus V(C_1 \cup \dots \cup C_{i-1})$, is at least $200r\log n$, and the average degree in the majority colour (say blue) is at least $200\log n \ge 100\ell_i$. As $\ell_i \ge \log n$, we can use \Cref{cor:exact-length-cycles} to find a blue cycle $C_i$ of length $\ell_i$ in $X'$, as needed.

		Again, let $\C$ be the set of cycles $C, C_1, \dots, C_t$ and possibly a singleton in $X$ (if $n$ is odd). Then $\C$ contains at most $200r+3$ cycles, and $A = X \sm V(\C)$ and $B = Y \sm V(\C)$ satisfy $|A| = |B| = {n}/{2} - k - \ell$.  We set $H=G[A,B]$.

		\medskip
		In either of the cases, $H$ is obtained from $G[X,Y]$ by deleting at most $2k+2\log n+4$ vertices, so each of the degrees can decrease by at most this value compared to \ref{itm:min-deg} and \ref{itm:almost-min-deg}. Assuming $\mu<{1}/{700000}$ and $n>100000$, we have $2k+2\log n+4 \le 10800\mu n + 2\log n + 4 \le {n}/{64}$. Then it is easy to check that $H$ is a balanced bipartite graph on at least ${n}/{2}$ vertices, such that $\delta(H)\ge n/16 - 10800\mu n - n/64 \ge {n}/{32}$, and all but ${n}/{64}$ vertices have degree at least ${n}/{3}$. This $H$ is indeed a $(\mu,20\mu)$-\robmat graph of type 2.
	\end{proof}
	
\section{Sharpness for the minimum degree}
	In this section we prove \Cref{prop:degree-best-possible}, which shows that the minimum degree condition in \Cref{thm:main} is almost best possible.
	
    \begin{proof}[Proof of \Cref{prop:degree-best-possible}]
		Our construction is based on the following claim.
		
		\begin{claim} \label{lem:largedeggirth}
			For every sufficiently large $n$, there is a graph on $n$ vertices with minimum degree at least $\log n$ that does not contain any proper cycle of length shorter than ${\log n}/{(4\log\log n)}$.
		\end{claim}

		\begin{proof}
			During this proof, all cycles will be proper.
			Our construction is probabilistic. We start with the Erd\H{o}s--R\'enyi random graph $G(n,p)$ on $n$ vertices, where any two vertices are connected by an edge independently with probability $p= {8\log n}/{n}$. Next, we take a maximal collection $\C$ of edge-disjoint cycles of length less than $k =  {\log n}/{(4\log\log n)}$. By maximality, $G'=G(n,p) \sm E(\C)$ has no cycles whose length is shorter than $k$. So it is enough to show that $\delta(G')\ge \ell=\log n$ with positive probability.
			
			To see this, first note that $\deg_{G(n,p)}(v) \sim \Bin(n-1,p)$ for every vertex $v$, so we can use the Chernoff bounds to bound the probability that the degree is small. Let $\calA$ be the event that some vertex of $G(n,p)$ has degree less than $3\log n$. By \Cref{lem:che} (with $a=5/8$) and a union bound over the vertices, we get $\Pr[\calA] \le n \cdot e^{-(25/16)\log n}< {1}/{3}$ for large enough $n$.
			
			Now for every vertex $v$ of $G(n,p)$, let $\calB_v$ be the event that $v$ is incident to at least $\ell$ pairwise edge-disjoint cycles of length shorter than $k$. Note that there are at most $n^{(k_1-1)+\dots + (k_{\ell}-1)}$ potential sets of $\ell$ edge-disjoint cycles of lengths $k_1,\dots,k_{\ell}$ incident to $v$, and each of them is a subgraph of $G(n,p)$ with probability $p^{k_1+\dots+k_{\ell}}$. Hence
			\[  
			\Pr[\calB_v] \le \sum_{3 \le k_1, \dots, k_{\ell} < k} n^{k_1 + \dots + k_{\ell} - \ell}p^{k_1 + \dots + k_{\ell}} \le \left( \frac{1}{n} \cdot \sum_{k'<k}  (np)^{k'} \right)^{\ell} \le \left( \frac{(np)^k}{n} \right)^{\ell} \le \left( \frac{1}{\sqrt{n}} \right)^{\log n}
			\]
			using $(np)^k= (8\log n)^{\log n/(4 \log\log n)}\le \sqrt{n}$. This means that the probability that $\calB=\bigcup \calB_v$ holds is at most $n\cdot n^{-(\log n)/2}< {1}/{3}$ for large enough $n$.
			
			So with positive probability, neither $\calA$ nor $\calB$ occur. But then every vertex $v$ is incident to at least $3\log n$ edges in $G(n,p)$, and at most $2\log n$ of those can appear in $\C$. This implies $\delta(G')\ge \log n$, as needed.
		\end{proof}
		
		Let $A$ and $B$ be disjoint sets, such that $|B| =  {n}/{2} +  {\log n}/{(16\log\log n)}$, and $|A| = n - |B|$. Let $G$ be a graph with vertex set $A \cup B$, where all $A$-$B$ edges are present and are red, and $G[B]$ is a blue graph provided by \Cref{lem:largedeggirth}. So $G[B]$ has minimum degree at least $\log |B| \ge \log n - 2 \ge  {\log n}/{(8\log\log n)}$, and it induces no proper cycle shorter than $ {\log |B|}/{(4\log\log |B|)} \ge \log n / (8 \log \log n)$.
		
		Note that $G$ has minimum degree at least $ {n}/{2} +  {\log n}/{(16\log\log n)}$. Also, every red cycle in $G$ is either a singleton, or it covers an equal number of vertices in $A$ and $B$. Moreover, every blue cycle is either a singleton, an edge, or has length at least $ {\log n}/{(8\log\log n)}$.
		
		Let $\C$ be a collection of vertex-disjoint monochromatic cycles covering all vertices of $G$. If $\C$ contains a proper blue cycle (of length at least $ {\log n}/{(8\log\log n)}$), then the remaining cycles of $\C$ must cover at least $ {\log n}/{(16\log\log n)}$ more vertices in $A$ than in $B$. But $A$ is independent, so this is only possible if $\C$ contains at least $ {\log n}/{(16\log\log n)}$ singletons. So $\C$ cannot contain any proper blue cycle. But then, as $\C$ covers $ {\log n}/{(16\log\log n)}$ more vertices in $B$ than in $A$, it must contain at least $ {\log n}/{(32\log\log n)}$ singletons or blue edges. Hence, in any case, $\C$ consists of at least $ {\log n}/{(32\log\log n)}$ cycles, as desired.
	\end{proof}

\section{Sharpness for the number of cycles}

    The goal of this section is to prove \Cref{prop:cyclelowerbound}, which shows that the number of cycles we use to partition the vertices is best possible, up to a constant factor.

    We start with some preliminary lemmas. 
    \begin{lem} \label{Lemma_min_degree_subgraph}
        Every $n$-vertex graph $G$ with $e(G)\ge (1-\eps^2)  {n^2}/{2}$ has a subgraph $H$ with $\delta(H)\ge (1-2\eps)n$.
    \end{lem}
    \begin{proof}
        Let $S$ be the set of vertices $v$ in $G$ with $\deg(v)\le (1-\eps)n$. We have $2e(G)\le (n-|S|)n+|S|(1-\eps)n = n^2-\eps |S|n$ which combined with $e(G)\ge {(1-\eps^2) n^2}/{2}$ gives $|S|\le \eps n$. Then $H=G\setminus S$ is a graph with $\delta(H)\ge (1-\eps)n-|S|\ge (1-2\eps)n$.
    \end{proof}

    Our argument uses the following theorem.
    \begin{thm}[Gy\'arf\'as--S\'ark\"ozy~\cite{GS12-rainbow}] \label{thm:gyarfsark}
        Every properly edge-coloured graph $G$ on $n$ vertices with $\delta(G)\le  {n}/{2}$ has a rainbow matching of size $\delta(G)-2\delta(G)^{2/3}$.
    \end{thm}

    \begin{cor} \label{Corollary_rainbow_matching}
        For $\eps>0$ and large enough $n$, every properly edge-coloured $n$-vertex graph $G$ with $\delta(G)\ge  {n}/{2}$ has a rainbow matching of size $(1-\eps) {n}/{2}$.
    \end{cor}
    \begin{proof}
        Delete edges from $G$ to get a spanning subgraph $H$ with $\delta(H)= {n}/{2}$, and apply \Cref{thm:gyarfsark} to $H$. We get a rainbow matching of size $ {n}/{2}-2({n}/{2})^{2/3}\ge (1-\eps) {n}/{2}$ (for sufficiently large $n$).
    \end{proof}

    \begin{lem}\label{Lemma_rainbow_matching}
        For $\eps>0$ and large enough $n$, every properly edge-coloured $n$-vertex graph $G$ with $e(G)\ge (1-\eps^2) {n^2}/{2}$ has a rainbow matching of size $(1-3\eps) {n}/{2}$.
    \end{lem}
    \begin{proof}
        Apply \Cref{Lemma_min_degree_subgraph} to get a subgraph $H$ on $m$ vertices with $\delta(H)\ge (1-2\eps)n$.
        Apply \Cref{Corollary_rainbow_matching} to $H$ (using that $\delta(H) \ge m/2$) in order to get a rainbow matching that has size $(1-\eps) {m}/{2}\ge (1-\eps) \delta(H) / 2 \ge (1 - \eps) (1 - 2\eps)n/2 \ge (1 - 3\eps)n / 2$.
    \end{proof}

    The following lemma is a bipartite version of the theorem we are aiming for. 
    \begin{lem}\label{Lemma_bipartite_construction}
        For any $\eps>0$ and sufficiently large $r$, there is an $r$-edge-coloured bipartite graph $G$ with parts $X$ and $Y$ such that
        \begin{enumerate}[\upshape (a)]
            \item \label{itm:bipconstrdeg} 
				$\deg(x)\ge (1-3\eps)|Y|$ for all $x\in X$, and
            \item \label{itm:bipconstrcov} 
				$X$ cannot be covered by fewer than $ {\eps^2r^2}/{4}$ monochromatic components in $G$.
        \end{enumerate}
    \end{lem}
    \begin{proof}
        Let $Y$ be a set of size $r$. Let $K_Y$ be an auxiliary  properly $r$-edge-coloured complete graph on vertex set $Y$.
        Let $X$ be the set of rainbow matchings in $K_Y$ of size $(1-3\eps) {r}/{2}$. 

        The graph $G$ is defined as follows: for any $x\in X$ and $y\in Y$, we add the edge $xy$ to $G$ in colour $i$ if the rainbow matching $x$ of $K_Y$ contains a colour-$i$ edge through $y$. If the rainbow matching $x$ does not contain any edge through $y$, then $xy$ is not present in $G$.

        To see that \ref{itm:bipconstrdeg} holds, notice that every $x\in X$ is connected to all the vertices of $Y$ that appear in the rainbow matching $x$ of $K_Y$. Since the rainbow matching $x$ has size $(1-3\eps) {r}/{2}$, we get $\deg(x)\ge (1-3\eps)|Y|$.

        Let $uv$ be an $i$-coloured edge of $K_Y$, and let $X_{uv}\subset X$ be the set of rainbow matchings containing $uv$. We claim that $T_{uv}=\{u,v\}\cup X_{uv}$ is an $i$-coloured component of $G$. Indeed, $u$ and $v$ are only adjacent to $X_{uv}$ in colour $i$ because $K_Y$ is properly coloured. Also, the matchings in $X_{uv}$ contain no $i$-coloured edges other than $uv$ because they are rainbow. 
        This shows that every monochromatic component of $G$ is either of the form $T_{uv}$ or is a singleton.

        Let $T_1, \dots, T_k$ be any family of $k\le  {\eps^2r^2}/{4}$ monochromatic components in $G$. We will find a vertex in $X$ that is not covered by these monochromatic components. 
        Using the previous paragraph, we may assume that the component $T_i$ has form $T_{u_iv_i}$ for some edge $u_iv_i\in K_Y$. Consider $H=K_Y\setminus\{u_1v_1, \dots, u_kv_k\}$. Then $e(H)\ge (1-\eps^2) {r^2}/{2}$, so by \Cref{Lemma_rainbow_matching}, $H$ has a rainbow matching $M$ of size $(1-3\eps) {r}/{2}$.
        This $M$ thus corresponds to a vertex $x_M\in X$. However, as $M$ does not contain the edge $u_iv_i$ for any $i$, the vertex $x_M$ does not belong to any $T_i$. In other words, $T_1, \dots, T_k$ do not cover the vertex $x_M$, establishing \ref{itm:bipconstrcov}.
    \end{proof}

    We are now ready to prove the main result of this section.

    \begin{proof}[Proof of \Cref{prop:cyclelowerbound}]
        Let us apply \Cref{Lemma_bipartite_construction} with $r-1$ colours to obtain an $(r-1)$-edge-coloured bipartite graph with parts $X$ and $Y$ satisfying \ref{itm:bipconstrdeg} and \ref{itm:bipconstrcov}.
        To construct $G$ from $H$, we blow up all vertices in $Y$ by $n'=n'(\eps,r)$ vertices, and add all the edges inside $Y$ in some previously unused colour. Formally, we introduce a set of vertices $V_y$ of size $n' = {|X|}/{\eps}$ for every $y\in Y$, and set $Y'=\bigcup_{y\in Y} V_y$. The vertices of $G$ are $V(G)=X\cup Y'$. For any $x\in X$ and $v\in V_y$, the edge $xv$ is present in $G$ in colour $i$ precisely when the edge $xy$ is present in $H$ in colour $i$. For $u,v\in Y'$, the edge $uv$ is present in $G$ in colour $r$.

        Let $n$ denote the number of vertices in $G$. Then $n= |X|(1+ {(r-1)}/{\eps})$, so in particular, $|X|\le \eps n - 1$. This together with the definition of $G$ implies that for every vertex $y\in Y'$, we have $\deg_G(y)\ge |Y'| - 1 \ge (1-\eps)n$. Using \ref{itm:bipconstrdeg} we also get that for every $x\in X$,
        \[ 
			\deg_G(x) = \deg_H(x) n' \ge (1-3\eps)|Y|n'= (1-3\eps)|Y'| \ge (1-3\eps)(1-\eps)n \ge (1-4\eps)n. 
		\]

        As $G$ was constructed from $H$ by blowing up $Y$ and then adding some edges in it using a new colour, we see that every monochromatic component of $G$ touching $X$ is of the form $\{x\in T \cap X\}\cup \{v\in V_y:y\in T\cap Y\}$ for some monochromatic component $T$ of $H$. But then any covering of $G$ by fewer than $ {\eps^2(r-1)^2}/{4}$ monochromatic components would give a covering of $X$ in $H$ by fewer than $ {\eps^2(r-1)^2}/{4}$ monochromatic components, contradicting \ref{itm:bipconstrcov}.
    \end{proof}

\subsection*{Acknowledgements}

We would like to thank the anonymous referees for their helpful comments and Louis DeBiasio for stimulating discussions.


\providecommand{\bysame}{\leavevmode\hbox to3em{\hrulefill}\thinspace}
\providecommand{\MR}{\relax\ifhmode\unskip\space\fi MR }
\providecommand{\MRhref}[2]{%
	\href{http://www.ams.org/mathscinet-getitem?mr=#1}{#2}
}
\providecommand{\href}[2]{#2}

\appendix

	\section{Proofs of preliminary results} \label{appendix:proofs-prelims}

		\subsubsection*{Proof of \Cref{lem:inner-paths}}
			\begin{proof} 
				The following claim is the main ingredient of the proof.
				\begin{claim} \label{claim:hamilton}
					Let $W_1 \subseteq V_1$ and $W_2 \subseteq V_2$ satisfy $|W_1|, |W_2| \ge n/10$ and $\delta(G[W_1, W_2]) \ge 3\eps n$. Then $G[W_1, W_2]$ has a cycle of length $2 \min\{|W_1|, |W_2|\}$.
				\end{claim}
				\begin{proof} 
					Let $P = u_1 \ldots u_{2\ell}$ be a longest path of even order in $H := G[W_1, W_2]$, where $u_1 \in W_1$ and $u_{2\ell} \in W_2$. Note that $\ell \ge 2\eps m$ by the minimum degree assumption. 
					
					We claim that $H$ has a cycle of length $2\ell$. 
					To see this, we will show that one of the following holds: 
					\begin{itemize} 
						\item
							$u_1 u_i$ is an edge and $u_{2\ell}, u_{i - 2}$ have a common neighbour $u \notin V(P)$, for some $i \in [2\ell]$,
						\item
							$u_i u_{2\ell}$ is an edge and $u_1, u_{i+2}$ have a common neighbour $u \notin V(P)$, for some $i \in [2\ell]$,
						\item
							$u_1 u_i$, $u_j u_{2\ell}$ and $u_{i-1} u_{j+1}$ are edges, for some $1 \le i < j \le 2\ell$,
						\item
							$u_1 u_j$, $u_i u_{2\ell}$ and $u_{i+1} u_{j+1}$ are edges, for some $1 \le i < j \le 2\ell$.
					\end{itemize}
					Note that in each of these cases, a cycle of length $2\ell$ exists. Indeed, here are the appropriate cycles: $(u_1 \ldots u_{i-2} u u_{2\ell} \ldots u_i)$, $(u_{2\ell} \ldots u_{i+2} u u_1 \ldots u_i)$, $(u_1 \ldots u_{i-1} u_{j+1} \ldots u_{2\ell} u_j \ldots u_i)$, and $(u_1 \ldots u_i u_{2\ell} \ldots u_{j+1} u_{i+1} \ldots u_j)$.

					For a set $S \subseteq V(P)$, define $S^- = \{u_{i-1} : u_i \in S\}$ and $S^+ = \{u_{i+1} : u_i \in S\}$. We now prove that one of the above four conditions holds. 
					
					Note that either $u_1$ does not have neighbours in $W_2 \setminus V(P)$ or $u_{2\ell}$ does not have neighbours in $W_1 \setminus V(P)$; without loss of generality, the latter holds. 

					Let $X$ be the set of neighbours of $u_1$ that are not in $P$, let $Y$ be the set of neighbours of $u_1$ in $P$, and let $Z$ be the set of neighbours of $u_{2\ell}$. If $|X| \ge \eps m$, then by regularity there is an edge between $X$ and $Z^{++}$, as required for the second item above. We now assume that $|X| \le \eps n$, so $|Y| \ge 2\eps n$. Let $\alpha \in [2\ell]$ be such that the sets $Y_1 = Y \cap \{u_1, \ldots, u_{\alpha}\}$ and $Y_2 = Y \cap \{u_{\alpha}, \ldots, u_{2\ell}\}$ each have size at least $\eps n$. Write $Z_1 = Z \cap \{u_1, \ldots, u_{\alpha - 1}\}$ and $Z_2 = Z \cap \{u_{\alpha + 1}, \ldots, u_{2\ell}\}$. If $|Z_1| \ge \eps n$, then by regularity there is an edge between $Z_1^-$ and $Y_2^+$, as required for the third item. Otherwise, $|Z_2| \ge \eps n$, and, again by regularity, there is an edge between $Z_2^+$ and $Y_1^+$, as required for the fourth item.

					We have established that there is a cycle $C$ of length $2\ell$; abusing notation slightly, write $C = (u_1 \ldots u_{2\ell})$.
					We now show that $\ell = \min \{|W_1|, |W_2|\}$. Suppose otherwise. 

					First, suppose that there is a vertex $v \notin V(C)$ with at least $\eps n$ neighbours outside of $C$; let $A$ be the set of these neighbours. By regularity, there is an edge between $A$ and $C$. It follows that there is a path of length $2\ell + 2$, a contradiction to the choice of $\ell$. 
					Next, we may assume that every vertex not in $C$ has at least $2\eps n$ neighbours in $C$. Let $v \in W_1 \setminus V(C)$ and $w \in W_2 \setminus V(C)$. Let $A$ and $B$ be the neighbourhoods of $v$ and $w$ in $C$, respectively. By regularity, there is an edge between $A^-$ and $B^-$; let $i, j$ be such that $u_i \in A$, $u_j \in B$, and $u_{i-1} u_{j-1}$ is an edge. Then $v u_i \ldots u_{j-1} u_{i-1} \ldots u_j w$ is a path of length $2\ell + 2$, a contradiction to the choice of $\ell$. 
				\end{proof}

				Let us now see how to deduce \Cref{lem:inner-paths} from \Cref{claim:hamilton}. Assume, without loss of generality, that $|U_1| \le |U_2|$, and write $m = |U_1|$.

				Consider the first part. For $i \in \{1, 2\}$, let $A_i\subseteq U_i$ be a set of $\ceil{\eps n}$ neighbours of $v_{3-i}$, and let $W_i$ be the set of vertices $u \in U_i \setminus A_i$ that have a neighbour in $A_{3-i}$ and have degree at least $7\eps n$ in $G[U_1, U_2]$. By regularity, $|W_i| \ge |U_i| - |A_i| - 2\eps n \ge |U_i| - 4\eps n \ge (1 - 24\eps)m$. It follows that $G[W_1, W_2]$ has minimum degree at least $3\eps n$. Thus, by \Cref{claim:hamilton}, there is a path $u_1 \ldots u_{2\ell}$ in $G[W_1, W_2]$, for any $\ell$ with $1 \le \ell \le \min\{|W_1|, |W_2|\}$; without loss of generality $u_1 \in W_1$ and $u_{2\ell} \in W_2$. By choice of $W_1, W_2$, there are vertices $x_1 \in A_1$ and $x_2 \in A_2$ such that $v_1 x_2 u_1 \ldots u_{2\ell} x_1 v_2$ is a path. It follows that there is a $v_1$-$v_2$-path of order $2k$ for every $3 \le k \le (1 - 24\eps)m$. To complete the proof of the first part, observe that, by regularity, there is an edge between $A_1$ and $A_2$, implying that there is a $v_1$-$v_2$-path of order $4$.

				Now we prove the second part. Note that it suffices to prove the existence of a $v_1$-$v_2$-path of order $2(k+1)$ with $(1 - 24\eps)m \le k \le m$. Let $U_1'$ be a subset of $U_1$ of size $k$, chosen uniformly at random among such sets. Then, with high probability, $G[U_1', U_2]$ has minimum degree at least $3\eps n$, and $U_1'$ contains at least $\eps n$ neighbours of $v_2$. It follows from \Cref{claim:hamilton} that there is a cycle $C = (u_1 \ldots u_k)$ in $G[U_1', U_2]$. Let $X$ be the set of neighbours of $v_1$ in $C$, let $Y$ be the set of neighbours of $v_1$ outside of $C$, and let $Z$ be the set of neighbours of $v_2$ in $C$. By regularity, there is an edge between $Z^+$ and $X^+$ or between $Z^{++}$ and $Y$ (where $X^+$, $Z^+$ and $Z^{++}$ are defined as in the proof of \Cref{claim:hamilton}). One can easily check that in either case, a $v_1$-$v_2$-path of order $2(k+1)$ exists.
			\end{proof}

		\subsubsection*{Proof of \Cref{prop:reducedgraph}}

			\begin{proof}
				\hfill
				\begin{enumerate}[\upshape (a)]
					\item 
						Note that $\deg_{G'}(v)\ge (c-rd-\eps)n$, so $v$ is adjacent to vertices from at least $(c-rd-\eps)n/|V_i|\ge (c-rd-\eps)m$ clusters in $G'$. By the definition of $G'$, this means that $x_i$ is adjacent to the corresponding vertices in $\G$.
					
					\item 
						We may assume $\eta+\eps<1$, as otherwise there is nothing to prove. Let $X$ be the set of vertices $x$ in $\G$ with $\deg_{\G}(x)<(c-rd-\eps)m$. By (a), we know that every vertex $v$ in the clusters corresponding to $X$ must have $\deg_G(v)<cn$. So if $|X|>(\eta+\eps) m$, then at least $|X||V_i|>(\eta+\eps)(1-\eps)n=(\eta +\eps(1-\eta-\eps))n>\eta n$ vertices $v\in V(G)$ have $\deg_G(v)<cn$, contradicting our assumption.
					
					\item 
						As $G'$ is obtained by deleting at most $(rd+\eps)n^2$ edges from $G \setminus V_0$, we know that $\bigcup_{x_i\in X} V_i$ induces at least $(c-rd-\eps)n^2$ edges in $G'$. But then there are at least $(c-rd-\eps)n^2/|V_i|^2\ge (c-rd-\eps)m^2$ pairs $\{x_i,x_j\}$ in $X$ with an edge between the corresponding clusters $V_i$ and $V_j$ in $G'$. By the definition of $G'$, this means that $X$ induces at least $(c-rd-\eps)m^2$ edges in $\G$. 
						\qedhere
				\end{enumerate}
			\end{proof}

		\subsubsection*{Proof of \Cref{lem:connecting-paths}}

			\begin{proof} 
				Let $i_1, \ldots, i_{\ell} \in [m]$ be a sequence such that $i_1 = i$, $i_2 = j$, $i_{\ell-1} = i'$ and $i_{\ell} = j'$; $x_{i_s} x_{i_{s+1}}$ is an edge of colour $c$ in $\G$ for $s \in [\ell-1]$; and $5 \le \ell \le m+2$ (it is easy to see that such a sequence exists).

				We claim that there exists a path $v_1 \ldots v_{\ell}$ such that $v_s \in V_{i_s} \setminus W$ for $s \in [\ell]$; $v_1 = v$ and $v_{\ell} = w$; and $v_s$ has typical degree in $V_{i_{s+1}}$ for $s \in [\ell - 3]$. Indeed, suppose that $v_1 \ldots v_{s-1}$ is a path with the required properties, for $2 \le s \le \ell - 2$. If $s \le \ell - 3$, as $v_{s-1}$ has typical degree in $V_{i_{s}}$, there is a neighbour $v_s \in V_{i_s} \setminus (W \cup \{v_1, \ldots, v_{s-1}, v_{\ell}\})$ of $v_{s-1}$ with typical degree in $V_{i_{s+1}}$, using \eqref{equ:typical-degree-in-reg-pair}. If $s = \ell - 2$, as $v_{\ell-3}$ and $v_{\ell}$ have typical degrees in $V_{i_{\ell-2}}$ and $V_{i_{\ell-1}}$, respectively, and by regularity, there are vertices $v_{\ell-2}$ and $v_{\ell-1}$ in $V_{i_{\ell-2}} \setminus (W \cup \{v_1, \ldots, v_{\ell-3}, v_{\ell}\})$ and $V_{i_{\ell-1}} \setminus (W \cup \{v_1, \ldots, v_{\ell-3}, v_{\ell}\})$, respectively, such that $v_{\ell-3} v_{\ell-2} v_{\ell-1} v_{\ell}$ is a path. The existence of the required path follows, proving \Cref{lem:connecting-paths}.
			\end{proof}
			
		\subsubsection*{Proof of \Cref{lem:EGP}}

			\begin{proof}

				An important ingredient is the following classic result of P\'osa (see~\cite{Lov79}).
			
				\begin{thm}[P\'osa] \label{thm:posa}
					The vertices of any graph $G$ can be covered with at most $\alpha(G)$ vertex-disjoint cycles.
				\end{thm}

				For $x \in B$ let $N_i(x)$ denote the set of vertices in $A$ adjacent to $x$ in colour~$i$. 
				Let $\{B_1, B_2, \dots ,B_r\}$ be a partition of $B$ such that for
				any $1 \le i \le r$ and $x \in B_i$, we have $|N_i(x)| \ge |A| /
				(100r)$ (clearly, such a partition exists).  
				Define a graph $G_i$
				on vertex set $B_i$ for every $i \in [r]$ as follows.  For $x,y \in B_i$, let $xy$
				be an edge of $G_i$ if and only if $|N_i(x) \cap N_i(y)| \ge |A| / (100^3r^3)$.
				We claim that the independence number of $G_i$ is at most $100r$ for $i
				\in [r]$.
				
				Indeed, suppose otherwise, and let $x_1, \dots, x_{100r+1} \in B_i$
				be pairwise non-adjacent. Then 
				\begin{align*}
					|A| & \ge 
					\left| \bigcup_{1 \le j \le 100r+1} N_i(x_j) \,\right| \\
					&\ge \sum_{1 \le j \le 100r+1}|N_i(x_j)|
						- \sum_{1\le j < k \le 100r+1} |N_i(x_j)\cap N_i(x_k)| \\
					&\ge |A| \left( \frac{100r+1}{100r} - \frac{\binom{100r+1}{2}}{100^3 r^3} \right) > |A|,
				\end{align*}
				a contradiction.
				
				So $\alpha(G_i)\le 100r$, and thus by \Cref{thm:posa} $G_i$ can be partitioned into a family $\C_i$ of at most $100r$ vertex-disjoint cycles. Using the definition of $G_i$, we can then greedily replace the edges $xy$ in each $\C_i$ with $i$-coloured paths $xwy$ (where $w\in N_i(x) \cap N_i(y)\subset A$), where each edge uses a distinct vertex $w$, to find at most $100r^2$ monochromatic vertex-disjoint proper cycles and edges in $G$ that cover $B$.
				Additionally, for every singleton $\{x\}$ in $\C_i$ we replace it by an edge $xw$ (where $w \in N_i(x)$), so that each singleton uses a distinct vertex $w$, which is also distinct from the vertices chosen previously.
			\end{proof}

		\subsubsection*{Proof of \Cref{prop:sample}}
		
			\begin{proof}
				Let $A$ be a random subset of $V\sm B$ where every vertex is included in $A$ independently with probability $p$. We will show that the event that $A$ satisfies all of the properties has positive probability.
				\begin{enumerate}[\upshape (a)]
					\item
						Note that $|A|\sim \Bin(|V\sm B|,p)$, and by assumption, $|B|<10pn+\eps n < n/3$. Hence, by \Cref{lem:che}, 
						\[ \Pr\left[ |A|<\frac{p}{2}n \right] \le \Pr\left[ |A|<\frac{3}{4}|V\sm B|p \right] < e^{-|V\sm B|p/32} \le e^{-np/48}.  \]
					\item
						Again, for any $i\in [m]$ fixed, $|A\cap V_i|\sim \Bin(V_i\sm B,p)$ so by \Cref{lem:che} and $|V_i\sm B|\ge  ({9}/{10})|V_i|> {n}/{(2m)}$,
						\[ \Pr\left[ |A\cap V_i| > 2p|V_i| \right] < \Pr\left[ |A\cap V_i| > 2p|V_i\sm B| \right] < e^{-|V_i\sm B|p/3} < e^{-np/(10m)} .\]
					\item
						Let $v \in V$ and $i\in[m]$ be such that $\deg(v,V_i)>30p|V_i|$.
						Here $\deg(v,A\cap V_i)\sim \Bin(\deg(v,V_i\sm B),p)$, and $\deg(v,V_i\sm B) \ge \deg(v,V_i) - |B\cap V_i| >  ({2}/{3})\deg(v,V_i)$.
						We can then apply \Cref{lem:che} to get
						\[ \Pr\left[  \deg(v,A\cap V_i)< \frac{p}{2}\deg(v,V_i) \right] < e^{-\deg(v,V_i)p/48} < e^{-|V_i|p^2/2} < e^{-np^2/(4m)} .   \]
					\item 
						Let $v\in V$ with $\deg(v,V\sm B)> {n}/{40}$.
						Like in the previous argument, $\deg(v,A)\sim \Bin(\deg(v,V\sm B),p)$, so by \Cref{lem:che} and as $\deg(v,V\sm B)>n/40$, 
						\[ \Pr\left[  \deg(v,A) <\frac{pn}{50}\right] \le \Pr\left[  \deg(v,A)< \frac{4p}{5}\deg(v,V\sm B) \right] < e^{-\deg(v,V\sm B)p/50} \le e^{-np/2000}.   \]
						But the argument for Property \ref{itm:sample-b} shows that $\Pr\left[ |A| >2pn \right] < m e^{-np/(10m)}$, so
						\[ \Pr\left[  \deg(v,A) < \frac{|A|}{100} \right] <  e^{-np/2000}+me^{-np/(10m)}. \]
				\end{enumerate}
				Now taking a union bound over all choices of $i\in [m]$ in (b) and (c) and all choices of $v\in V$ in (c) and (d), we get that $A$ satisfies all properties with probability at least $1-n^2e^{-np^2/(4m)}>0$. (Here we use that $p> {m\log n}/{\sqrt{n}}$ and in particular $m \ll n$). In particular, there is such a set $A$.
			\end{proof}

\end{document}